\documentclass[11pt,english]{article}

\usepackage[a4paper,
            left=0.67in,
            right=0.67in,
            top=1in,
            bottom=1in,
            footskip=.25in]{geometry}

\usepackage[utf8]{inputenc}
\usepackage[english]{babel}
\usepackage{pdfpages} 
\usepackage{graphicx}
\usepackage{siunitx}
\usepackage{hyperref}
\usepackage{url}
\usepackage{amsmath}
\usepackage{cleveref}
\usepackage{amssymb}
\usepackage[a4paper]{geometry}
\usepackage{amsthm}
\usepackage{cases}
\usepackage{yfonts}
\usepackage{tikz}
\usepackage{anyfontsize}
\usepackage{appendix}
\usepackage{enumitem}
\usepackage{amsfonts}
\usepackage{dsfont}
\usepackage{comment}
\usepackage{mathtools}
 
\usepackage{bbm}

\usepackage{tikz}

\newcommand{\br}{\mathbb{R}}

\newcommand{\bn}{\mathbb{N}}
\newcommand{\bt}{\mathbb{T}}

\newtheorem{thm}{Theorem}[section]
\newtheorem{lem}[thm]{Lemma}
\newtheorem{prop}[thm]{Proposition}

\newtheorem{rem}[thm]{Remark}

\numberwithin{equation}{section}

\usepackage{scalerel,stackengine}
\stackMath
\newcommand\reallywidehat[1]{%
\savestack{\tmpbox}{\stretchto{%
  \scaleto{%
    \scalerel*[\widthof{\ensuremath{#1}}]{\kern-.6pt\bigwedge\kern-.6pt}%
    {\rule[-\textheight/2]{1ex}{\textheight}}%WIDTH-LIMITED BIG WEDGE
  }{\textheight}% 
}{0.5ex}}%
\stackon[1pt]{#1}{\tmpbox}%
}

\title{Landau damping on the torus for the Vlasov-Poisson system\\ with massless electrons}

\author{
Antoine Gagnebin
  \thanks{ETH Z\"urich. Email: \textsf{antoine.gagnebin@math.ethz.ch}}
  \and
Mikaela Iacobelli
  \thanks{ETH Z\"urich. Email: \textsf{mikaela.iacobelli@math.ethz.ch}}
}
\date{\today}

\begin{document}

\maketitle

\begin{abstract}This paper studies the nonlinear Landau damping on the torus $\mathbb{T}^d$ for the Vlasov-Poisson system with massless electrons (VPME). We consider solutions with analytic or Gevrey ($\gamma > 1/3$) initial data, close to a homogeneous equilibrium satisfying a Penrose stability condition. We show that for such solutions, the corresponding density and force field decay exponentially fast as time goes to infinity. This work extends the results for Vlasov-Poisson on the torus to the case of ions and, more generally, to arbitrary analytic nonlinear couplings.
\end{abstract}

\section{Introduction}
\label{section Introduction}
In this paper, we study the following class of Vlasov-type systems on the torus $\bt^d$:
\begin{eqnarray}
\label{VP}
  \left \{
  \begin{array}{l}
  \partial_t f(t,x,v) + v \cdot \nabla_x f(t,x,v) + E (t,x)\cdot \nabla_v f(t,x,v)= 0, \\
  E(t,x) = - \nabla U(t,x), \qquad -\Delta U(t,x) + \beta U(t,x) + h(U)(t,x) =\rho(t,x)-1,\\
  \rho(t,x) = \int_{\mathbb{R}^d} f(t,x,v) \ dv,\\
  f(0, x, v) = f^0(x,v) \geq 0, \quad \int_{\mathbb{T}^d \times \mathbb{R}^d} f^0(x,v) \ dx \ dv = 1.
  \end{array}
  \right.
\end{eqnarray}
Such systems are a fundamental model of plasma physics and galactic dynamics (see, e.g. \cite{plasma_2003, Galactic_2008, Ryutov_1999, Villani_notes}).
The unknown $f(t,x,v)$ is the distribution function of particles at time $t$, position $x$, and velocity $v$, where $(t,x,v) \in \mathbb{R} \times \mathbb{T}^d\times \mathbb{R}^d$. We denote by $\rho$ the density of particles and $E$ the force field generated by the collective behaviour of the particles. Here, $h$ is an analytic function with some properties we will define below, while $\beta$ is a non-negative constant. 

With this general definition of the coupling, one could study several models for plasma physics and galactic dynamics as well as describe different mean-field type systems depending on the choice of $\beta$ and $h$.

Concerning plasma physics, a classical model for a non-relativistic plasma in the electrostatic regime is the Vlasov-Poisson system (VP), which corresponds to the case $\beta=h= 0$.
It describes the motion of electrons in plasma (i.e., $f(t,x,v)$ is the distribution function for the electrons) when we neglect collisions and consider the ions as a stationary background. 
Since (VP) is the most classical model for plasmas, vast literature proves the existence of global classical and weak solutions with various conditions on the initial data. 
In the whole space, the global existence and uniqueness of classical solutions of the Cauchy problem for the (VP) system was obtained by S. V. Iordanskii \cite{Iordanskii} in dimension one,  S. Ukai and T. Okabe \cite{Ukai_Okabe} in the two-dimensional case and independently by P.-L. Lions and B. Perthame \cite{Lions_Perthame}, and K. Pfaffelmoser \cite{Pfaffelmoser} in the three-dimensional case, (see also \cite{Bardos_Degond_2, Schaffer}). A. A. Arsenev \cite{Arsenev} proved a global existence of weak solutions in dimension three, E. Horst and R. Hunze \cite{Horst_Hunze} improved this result with weaker assumptions on the initial data (see also \cite{Bardos_Degond_1, Bardos_Degond_Golse}).
The adaptation to the periodic case was performed by J.  Batt and G. Rein \cite{Batt_Rein}. Later, C. Pallard \cite{Pallard} and Z. Chen and J. Chen \cite{Chen} improved the moment conditions for the existence. Finally, in \cite{Loeper}, G. Loeper proved uniqueness under the sole assumption of the density being bounded.

The preceding model can be extended to describe the evolution of ions with certain approximations. However, to achieve a more precise analysis of ions behaviour in the plasma, it is common in physics literature to assume that the electrons are close to thermal equilibrium. Indeed, although electron-electron collisions are neglected in the model above because of their rarity, they become relevant in the ions timescale. Therefore, it is reasonable to assume that the distribution of the electrons is the thermal equilibrium of a collisional kinetic model.
The Vlasov-Poisson model for massless electrons (VPME) --sometimes referred to as ionic Vlasov-Poisson--can then be derived asymptotically as the mass ratio between electrons and ions grows small. For more details on the massless limit, we refer to \cite{Bardos}, and for a more thorough introduction of this model, we refer to \cite{Griffin_Iacobelli_summary}.
The (VPME) system consists of a Vlasov equation coupled with a nonlinear Poisson equation modelling how the electric potential is generated by the distribution of the ions and the Maxwell-Boltzmann distribution of the electrons. In our notation, this corresponds to choosing $\beta=1$ and $h(U) = e^U - 1- U.$  

Due to the mathematical difficulties created by this nonlinear coupling, (VPME) has been less studied. However, F. Bouchut \cite{Bouchut} constructed weak solutions globally in time in the whole space in dimension three. In the one-dimensional setting, D. Han-Kwan and M. Iacobelli \cite{Han-Kwan_Iacobelli} proved the existence of global weak solutions for measure data with bounded first moment.
More recently, M. Griffin-Pickering and M. Iacobelli proved the global well-posedness of (VPME) in the torus in dimension two and three \cite{Griffin_Iacobelli_torus}, and in the whole three-dimensional space \cite{Griffin_Iacobelli_R}.  Furthermore, it is important to highlight that the theory for (VPME) models now incorporates stability and uniqueness results that have been developed in the context of quasi neutral limits. See for example \cite[Theorem 1.9]{Han-Kwan_Iacobelli} and \cite[Proposition 4.1]{Griffin_Iacobelli_stability} for stability results. Moreover, L. Cesbron and M. Iacobelli proved the well-posedness of (VPME) in bounded domains \cite{Cesbron_Iacobelli}.
See \cite{Griffin_Iacobelli_summary} for a review of the global well-posedness theory of (VPME). 

Another important model corresponds to selecting $\beta=1$ and $h = 0$, when (\ref{VP}) becomes the screened (VP) system. This system is a linear version of the coupling in the (VPME) system. That is, we have linearised the exponential $e^U$ by $1+U$. From a physical point of view, this approximation is valid as long as the electric energy is small compared to the kinetic energy.  The weak global existence of this system is discussed in \cite[Theorem 2.1]{HKD}. The proof of this result is an adaptation of A. A. Arsenev \cite{Arsenev}. For a more detailed discussion about this system, the interested reader can look into \cite[Section 1.1 and 1.2]{HKD_HDR}. All these results have been developed in the context of quasi neutral limits as well.

At this point, it is worth mentioning the link between the quasi-neutral limit and the large-time dynamics of Vlasov-type systems, which is the aim of this article. To briefly explain, the quasi-neutral limit involves introducing a small parameter $\varepsilon$ into equation (\ref{VP}) and considering the limit as $\varepsilon$ approaches zero. By employing a change of variables $(t,x,v) \mapsto \left( \frac{t}{\varepsilon}, \frac{x}{\varepsilon}, v \right)$, we can transition from the quasi-neutral Vlasov-type system to the Vlasov-type system with $\varepsilon = 1$. Consequently, establishing a uniform estimate over a time interval of size $O(1)$ in the quasi-neutral regime is essentially equivalent, in a formal sense, to proving this estimate over a time interval of size $O(\frac{1}{\varepsilon})$ for the Vlasov-type system. In simpler terms, the quasi-neutral limit can be interpreted as a problem concerning the asymptotic in-time behaviour. 
D. Han-Kwan discussed this aspect in his PhD Thesis \cite{HKD_thesis} and HDR \cite{HKD_HDR}, where he provides more details on this relation. For more information, interested readers can refer to \cite[Section 3.1.2]{HKD_thesis} in his PhD Thesis and \cite[Section 1.3 and Section 1.6]{HKD_HDR} in his HDR.

In the following, we assume $\beta\ge0$ and $h : (-R, R) \rightarrow \br $ to be analytic with analyticity radius $R$ and satisfying $h(x) = \mathcal{O}(x^{2})$ when $x$ tends to zero.  Since, as it will be clear in the following, we will consider small solutions of size $\varepsilon$, for any $h$ with analytic radius $R$, we can always choose $\varepsilon$ small enough with respect to $R$ such that $h(U)$ is well defined. Nevertheless, we note that in the (VPME) case with $h(x) = e^x -1 -x$, we have $R=\infty$. 

This paper studies the large-time behaviour of solutions with analytic or Gevrey initial data and close to a homogeneous equilibrium satisfying a Penrose stability condition. Namely, we look at solutions of the form
\begin{equation*}
	\mu(v) + f(t,x,v),
\end{equation*}
where $\mu(v)$ is a homogeneous equilibrium, and $f$ is assumed to be a small perturbation of $\mu$. 
We assume that the homogeneous equilibrium satisfies the following properties: 
\begin{itemize}
\item[(H1)] $\mu(v)$ is real analytic and for some constants $C>0$, $\theta_0 > 0$ and for all multi-index $j \in \bn^d$ such that $\vert j \vert \leq d$, we have
\begin{equation}
\label{Pmu1}
  \vert \partial_\eta^j \widehat{\mu}(\eta) \vert \leq C e^{-\theta_0 \vert \eta \vert},
\end{equation}
where $\widehat{\mu}$ denotes the Fourier transform of $\mu$.
\item[(H2)] $\mu (v)$ satisfies the following stability condition: there exists a small positive constant $\kappa_0$ such that
\begin{equation}
\label{stability_cond}
  \inf_{k \in \mathbb{Z}^d \setminus \{0\}; \Re \lambda \geq 0}{\left| 1 + \frac{\vert k \vert ^2}{\beta + \vert k \vert^2} \int_0^{\infty}  t\widehat{\mu}(kt) e^{-\lambda t} \ dt \right|} \geq \kappa_0 > 0,
\end{equation}
where $\lambda \in \mathbb{C}$ and $\Re \lambda$ is the real part of $\lambda$.
\item[(H3)] $\int_{\mathbb{R}^d} \mu (v) \ dv = 1.$
\end{itemize}

These three assumptions hold for a large class of equilibria. Indeed, any Maxwellian 
\begin{equation*}
	\mu(v) = e^{-\frac{\vert v \vert^2}{2}}
\end{equation*}
satisfies the assumptions above, or in dimension three or higher, every positive and radially symmetric equilibria also do. 

It is important to remark that the system \eqref{VP} may not be in general well-defined for all choices of $\beta$ and $h$. This is because the nonlinear coupling between $\rho$ and $U$ may not have a unique solution (namely, given $\rho$, there may be more than one function $U$ satisfying the elliptic PDE).
This non-uniqueness issue does not happen in the classical (VP) case ($\beta=h=0$) or for the screened (VP) case ($\beta=1$, $h=0$).
Even more,  for the nonlinear coupling of (VPME) ($\beta=1$, $h(U)=e^U-1-U$), the solution is unique,  as shown in \cite{Han-Kwan_Iacobelli} (see also Proposition 3.5 of \cite{Griffin_Iacobelli_torus}).
This is not a problem in our setting of general nonlinearities in the Poisson equation, provided we look for a small perturbation of equilibria. Indeed, when considering solutions of the form $\mu+f$ with $f$ small in a sufficiently strong norm (in our setting $f$ will be small in a suitable Gevrey norm), it is natural to impose that also $U$ should be small in some strong norm. This means that one is looking for solutions of 
\begin{equation*}
  -\Delta U + \beta U + h(U) = \int_{\mathbb{R}^d} \big(\mu(v) +f(t,x,v)\big)\ dv -1 = \int_{\mathbb{R}^d} f(t,x,v)\ dv,
\end{equation*}
with $ \int_{\mathbb{R}^d} f(t,x,v)\ dv$ small. Then this PDE corresponds to the Euler-Lagrange equation of a functional which is convex in a neighbourhood of $U\equiv0$--recall that $h(U)=\mathcal{O}(U^{2})$ by assumption, cp. \cite[Proposition 3.5]{Griffin_Iacobelli_torus}. Therefore, once we consider small solutions $U$, there is a unique minimum in a neighbourhood of the origin. 
In particular, if $f\equiv 0$ then $U\equiv 0$ is the unique solution.

As already mentioned before, we assume that at time $t=0$, the perturbation $f^0(x,v)$ lies in an analytic or Gevrey space. Since the total mass equals one, we have 
\begin{equation*}
  1 = \int_{\mathbb{T}^d \times \mathbb{R}^d} \big(\mu(v) + f^0(x,v)\big) \ dx \ dv = 1 + \int_{\mathbb{T}^d \times \mathbb{R}^d} f^0(x,v) \ dx \ dv,
\end{equation*}
therefore, we ask that the mass of the initial perturbation is zero. 
These considerations lead to the following perturbed nonlinear equation,
\begin{eqnarray}
\label{nlpvpme}
  \left \{
  \begin{array}{l}
  \partial_t f + v \cdot \nabla_x f + E[f] \cdot \nabla_v (f + \mu) = 0, \\
  E = - \nabla U, \qquad - \Delta U + \beta U + h(U)= \rho, \\
   \rho = \int_{\mathbb{R}^d} f \ dv, \\
   \int_{\mathbb{T}^d \times \mathbb{R}^d} f^0 \ dx \ dv = 0.
  \end{array}
  \right.
\end{eqnarray}
It is classical to start the analysis of the system above by looking at the following model
\begin{eqnarray}
\label{lpvpme}
  \left \{
  \begin{array}{l}
  \partial_t f + v \cdot \nabla_x f + E[f] \cdot \nabla_v \mu = 0, \\
  E = - \nabla U, \qquad -\Delta U +\beta U + h(U)= \rho, \\
   \rho = \int_{\mathbb{R}^d} f \ dv, \\
   \int_{\mathbb{T}^d \times \mathbb{R}^d} f^0 \ dx \ dv = 0.
  \end{array}
  \right.
\end{eqnarray}
This system \eqref{lpvpme} represents the linearized Vlasov equation around the equilibrium state $\mu$, coupled with the nonlinear Poisson equation. This model is established because, by focusing on this simplified version, we can more easily analyse its properties and obtain reliable estimates which can later be used when studying the more complex system described by \eqref{nlpvpme}.

We note that the system \eqref{lpvpme} is the linearised Vlasov equation around the equilibrium $\mu$ coupled with the non linear Poisson equation.  The motivation to define this model is that it's easier to first look at the asymptotic behaviour of this system and obtain good estimates that we can use later for the study on the system \eqref{nlpvpme}.

The goal of this paper is to prove collisionless relaxation, also known as Landau damping, for the nonlinear class of mean-field systems (\ref{nlpvpme}), as done for the classical (VP) system in \cite{CM_CV, BMM_13, GNR}.
We will not delve into a historical discussion about Landau damping (see, for example, \cite{CM_CV} and references therein). Still, to put our result into context, let us mention that in 1946, L. Landau formally studied the linearized (VP) system (i.e., (\ref{lpvpme}) with $\beta=h=0$) around a spatially homogeneous Maxwellian equilibrium on the torus, and he showed that the electric field decays exponentially fast \cite{Landau}. Later, O. Penrose \cite{Penrose} extended the result for general spatially homogeneous equilibria. 
The nonlinear analogue of Landau damping had been a longstanding problem in the theory until, in the celebrated work \cite{CM_CV}, C. Mouhot and C. Villani proved Landau damping for the nonlinear (VP) system on the torus for analytic and Gevrey data with index $\gamma$ close to one. They relied on some suitable families of analytic or Gevrey norms, measuring regularity by comparison with solutions of the free transport equation. Moreover, they employed a sophisticated use of Eulerian and Lagrangian coordinates combined with a global-in-time Newton approximation scheme. They explained the damping phenomenon in terms of transfer of regularity between kinetic and spatial variables and showed that phase mixing is the driving mechanism of relaxation.  

The mixing mechanism behind Landau damping also appears in fluid dynamics. Indeed, in \cite{BM} J. Bedrossian and N. Masmoudi proved asymptotic stability of shear flows close to the planar Couette flow in the $2$D inviscid Euler equations on $\mathbb{T}\times\mathbb{R}$. This long-time stability phenomenon is called inviscid damping and is considered the hydrodynamic analogue of Landau damping.
Later, the argument in \cite{CM_CV}  has then been simplified and extended to a wider class of Gevrey initial data by J. Bedrossian, N. Masmoudi, and C. Mouhot in \cite{BMM_13}. Indeed, the authors combined the ideas from  \cite{CM_CV} and the novel ideas that arose in the study of inviscid damping in $2$D Euler \cite{BM}. In particular, they rely on a paraproduct decomposition and controlled regularity loss to replace the Newton iteration scheme. Recently,  E. Grenier, T. Nguyen, and I. Rodnianski in \cite{GNR} further simplified the proof in \cite{BMM_13}. 

The assumption of analytic or Gevrey initial data is necessary. In \cite{JB}, J. Bedrossian showed that the result of Mouhot and Villani \cite{CM_CV} is false if we only assume Sobolev regularity for the initial condition. However, I. Tristani \cite{IT} showed first that adding a collision operator to the linear (VP) system allows us to consider Sobolev data to see Landau damping.  J. Bedrossian \cite{JB_collision} successively proved Landau damping for Sobolev data when considering collisions of particles with a Fokker-Planck operator for the nonlinear (VP) system. Recently,  S. Chaturvedi, J. Luk, and T. Nguyen \cite{Nguyen_collision} proved a similar result but for the Landau collision operator. These three results concern equilibrium of the form of a global Maxwellian and for a weakly collisional regime.

Since extending the theory to the unconfined case is physically relevant, several other works treated Landau damping on the whole space. First, J. Bedrossian, N. Masmoudi, and C. Mouhot proved Landau damping for Sobolev data in $\br^d$, $d \geq 3$ for the screened (VP) system \cite{BMM_16}. Then their result was improved by D. Han-Kwan, T. Nguyen, and F. Rousset in \cite{HKNR}; and very recently in \cite{huang_sharp_2022}, L. Huang, Q.-H. Nguyen, Y. Xu gave a better decay in time for the density in $\br^d$, $d \geq 3$ for a large class of couplings, including screened (VP) and (VPME). The same authors also treated the two-dimensional case in \cite{huang_2d}. We also mention the result of R. H\"ofer and R. Winter \cite{RHRW}, who consider the screened (VP) system coupled to the motion of a point charge. They showed that Landau damping  occurs in regions where the point charge has already passed. For the general (VP) system,  J. Bedrossian, N. Masmoudi, and C. Mouhot studied the linear Landau damping for homogeneous Maxwellian equilibrium in \cite{BMM_22_linear_VP}.  D.  Han-Kwan, T.  Nguyen, and F.  Rousset treated in  \cite{HKNR_linear_VP} the  linear (VP) equation for general homogeneous equilibria.  Recently,  in \cite{IPWW},  A.  Ionescu,  B.  Pausader and X.  Wang, and K.  Widmayer gave a proof of the nonlinear Landau damping for Poisson equilibrium. It is worth mentioning also a recent result concerning the analogue problem for the relaxation of stellar systems \cite{gravitational_LD}. The authors M. Had\v{z}i\'{c}, G. Rein, M. Schrecker, and C. Straub proved that the Landau damping occurs for solutions of the linearized gravitational (VP) system under certain conditions.

Our result concerns Landau damping for analytic or Gevrey initial data on the torus $\bt^d$ for a large class of couplings, including the (VPME) case.

\subsection{Notation}
\label{subsection Notation}

Let $k \in \mathbb{Z}^d$ and $\eta \in \mathbb{R}^d$, we denote the Fourier coefficient of $\rho(t,x)$ by 
\begin{equation*}
  \widehat{\rho}_k(t) := \widehat{\rho}(t,k) = \int_{\mathbb{T}^d} \rho (t,x) e^{-ik \cdot x} \ dx,
\end{equation*}
and we have the usual formula
\begin{equation}
\label{reconstruction formula 1}
	\rho (t,x) =  \frac{1}{(2 \pi)^d }  \sum_{k \in \mathbb{Z}^d} \widehat{\rho}_k(t) e^{ik \cdot x}.
\end{equation}
We write the Fourier transform of $f(t,x,v)$ by 
\begin{equation*}
  \widehat{f}_{k,\eta}(t) := \widehat{f}(t,k,\eta) = \int_{\mathbb{T}^d \times \mathbb{R}^d} f (t,x,v) e^{-ik \cdot x} e^{-i \eta \cdot v} \ dx \ dv,
\end{equation*}
with the reconstruction formula
\begin{equation}
\label{reconstruction formula 2}
	f(t,x,v) = \frac{1}{(2 \pi)^{2d} } \sum_{k \in \mathbb{Z}^d} \int_{\mathbb{R}^d} \widehat{f}_{k,\eta}(t)  e^{ik \cdot x} e^{i \eta \cdot v} \ d\eta.
\end{equation}
We recall as well the Laplace transform of a function $\phi \in L^2(\mathbb{R}_+)$,
\begin{equation}
	\mathcal{L}[\phi] (\lambda) = \int_0^{\infty} \phi(t) e^{-\lambda t} \ dt.
\end{equation}
It is well-defined for any complex value $\lambda$ with $\Re \lambda > 0$.
The Japanese bracket is written as follows: $\langle k, \eta \rangle = \sqrt{1 + \vert k \vert^2 + \vert \eta \vert^2}$.  We use analytic or Gevrey norms to control the decay of the electric field. Our norms are defined via the so-called “generator functions" introduced in \cite{EG_TN} (see also \cite{BMM_13,GNR}), that measure the analyticity or Gevrey regularity. 

More precisely, following \cite{GNR}, let $z \geq 0$ be the analyticity radius, $ \gamma \in (0, 1]$ the Gevrey index, $j \in \mathbb{N}^d$ a multi-index, $\sigma > 0$ and $\alpha< \frac{1}{2}$ such that $\sigma - \alpha > d$. Then we define the generator functions
\begin{equation}
\label{GenF}
	F[\rho](t,z) := \sup_{k \in \mathbb{Z}^d \setminus \{0\}} e^{z \langle k, kt \rangle^{\gamma}} \vert \widehat{\rho}_k (t) \vert \langle k, kt \rangle^{\sigma} \frac{1}{\vert k \vert^{ \alpha }} ,
\end{equation}
and
\begin{equation}
\label{GenG}
	G[f(t)](z) := \sum_{\vert j \vert \leq d} \sum_{k \in \mathbb{Z}^d} \int_{\mathbb{R}^d} e^{2z\langle k, \eta \rangle^{\gamma} } \vert \partial_{\eta}^j \widehat{f}_{k,\eta} (t) \vert^2 \langle k, \eta \rangle ^{2 \sigma} \ d\eta.
\end{equation}
For simplicity, we define $A_{k,\eta}:=e^{z \langle k, \eta \rangle^{\gamma}}  \langle k, \eta \rangle^{\sigma}$. In this paper, $C$ denotes a generic constant that may change from line to line.

\subsection{Main result}
\label{subsection Main result}
Now we state our main result.
\begin{thm} \ \\
\label{Thm1}
Consider the system \eqref{nlpvpme}, with
$\beta\ge0$ and $h : \br \rightarrow \br $ analytic and satisfying $h(x) = \mathcal{O}(x^{2})$ when $x$ tends to zero.
Let $\mu$ be a homogeneous equilibrium that satisfies the hypotheses (H1)-(H3).
Let $\lambda_1 >0$ and $\gamma \in (\frac{1}{3}, 1]$. Then there exists $\varepsilon > 0$ such that for any initial data $f^0$ of mean zero and satisfying
\begin{equation}
\label{IC}
	G[f^0](\lambda_1) \leq \varepsilon,
\end{equation}
Landau damping occurs for the system (\ref{nlpvpme}). More precisely, the force field $E$ and the density $\rho$ of (\ref{nlpvpme}) go to zero exponentially fast as $t$ goes to $+\infty$ in every $C^k (\bt^d)$ norm, with $k \in \bn$. Moreover, there exists a limit $f_\infty (x,v)$ such that the solution scatters to free transport as follows: given $ \lambda_0 \leq \frac{\lambda_1}{4}$, let  $z \leq \frac{\lambda_0}{2}$ for $\gamma \in (\frac{1}{3}, 1)$ and $z \leq \min \{ \frac{ \lambda_0}{2}, \frac{\theta_0}{2} \}$ for $\gamma = 1$ and where $\theta_0$ is defined in \eqref{Pmu1}. Then
\begin{equation}
\label{scaterring}
	G[f(t,x+vt, v) - f_\infty(x,v)](z) \leq C \varepsilon e^{- 2 (\lambda_0 - z) \langle t \rangle^{\gamma} }.
\end{equation}
Finally, the solution $f(t)$ converges weakly in $L^2 (\bt^d \times \br^d)$ as $t$ goes to $+\infty$ to the spatial average of $f_\infty$,
\begin{equation*}
	 \langle f_\infty (v) \rangle_x = \int_{\bt^d} f_\infty(x,v) \ dx.
\end{equation*}
\end{thm}

Our proof is inspired by the argument of E.  Grenier, T. Nguyen, and I. Rodnianski in \cite{GNR} for the classical (VP) system.
Here, the additional difficulty comes from the nonlinear term $h(U)$ in the Poisson equation. Indeed, the argument in  \cite{GNR} crucially relies on the linear relation between $U$ and $\rho$ that now fails. To emphasise this additional difficulty, we also note that in the (VPME) system, we lost the property that the electric field is given by a Fourier multiplier operator applied to the density $\rho$. More precisely, we define $X$ the space of analytic or Gevrey functions and the operator $\mathcal{T}$ as
\begin{align*}
	\mathcal{T} : \ X & \longrightarrow X \\
	 \rho & \longmapsto \mathcal{T}\rho = E.
\end{align*}
Then, using the Poisson equation of (VP), the Fourier transform of the operator $\mathcal{T}$ can be written with a multiplier as 
\begin{equation*}
	\reallywidehat{\mathcal{T}\rho} (k) = m(k) \widehat{\rho}_k = \frac{-ik}{\vert k \vert^2} \widehat{\rho}_k.
\end{equation*}
Similarly, for the screened (VP), the Fourier transform of the operator $\mathcal{T}$ is given by
\begin{equation*}
	\reallywidehat{\mathcal{T}\rho} (k) =  \frac{-ik}{1+\vert k \vert^2} \widehat{\rho}_k.
\end{equation*}
Due to the nonlinear term $h(U)$ in the Poisson coupling for (VPME), it is no longer possible to write the electric field in this way. Since we work on the Fourier side, this fact adds essential modifications to Landau damping analysis.

The paper is structured as follows: in the next section, we give some technical results needed for the rest of the paper. Then, in Section \ref{section Linear Landau damping}, we prove the linear Landau damping for the system (\ref{lpvpme}). We also show useful estimates on the linear system that we will use later for the nonlinear part. Finally, Section \ref{section Nonlinear Landau damping} is devoted to the study of nonlinear Landau damping and contains the proof of Theorem \ref{Thm1}.  Although our arguments are not dimension-dependent, we write the proof in the physical case $d=3$.

\section{Technical Lemmas}
\label{section Technical Lemmas}
In this section we collect some technical results that will be useful in the sequel.

\begin{lem} Let $F[\cdot]$ be defined as in (\ref{GenF}) with $\sigma > 3$ and $\alpha< \frac{1}{2}$ and such that $\sigma - \alpha > 3$. Then there exists a constant $C>0$ such that for any functions $\phi$ and $\psi$ with $ \widehat{\phi}_0 =  \widehat{\psi}_0 = 0$ there holds
\label{property_F}
\begin{equation*}
	F[\phi \psi](t,z) \leq C F[\phi] (t,z) F[\psi] (t,z).
\end{equation*}
\end{lem}

\begin{proof}
Let $\phi$ and $\psi$ be two functions and recall that $A_{k,\eta}=e^{z \langle k, \eta \rangle^{\gamma}}   \langle k, \eta \rangle^{\sigma}$, then we have
\begin{align*}
	F[\phi \psi](t,z) &= \sup_{k \in \mathbb{Z}^3 \setminus \{0\}} A_{k, kt} \vert \widehat{(\phi \psi)}_k (t) \vert \frac{1}{\vert k \vert^{ \alpha }} =  \sup_{k \in \mathbb{Z}^3 \setminus \{0\}} A_{k, kt} \frac{1}{\vert k \vert^{ \alpha }}  \Biggl| \sum_{ \substack{ \ell \in \mathbb{Z}^3 \setminus \{0\} \\ k \neq \ell }} \widehat{\phi}_l (t) \widehat{\psi}_{k - \ell} (t)  \Biggr|  \\
	  & =  \sup_{k \in \mathbb{Z}^3 \setminus \{0\}} A_{k, kt} \frac{1}{\vert k \vert^{ \alpha }}  \Biggl|  \sum_{ \substack{ \ell \in \mathbb{Z}^3 \setminus \{0\} \\ k \neq \ell }} \frac{A_{ \ell,  \ell t}}{\vert \ell \vert^\alpha}  \widehat{\phi}_l (t) \frac{A_{k - \ell, (k - \ell)t}}{\vert k - \ell \vert^\alpha}   \widehat{\psi}_{k - \ell} (t) A_{ \ell,  \ell t}^{-1}  A_{k - \ell, (k - \ell)t}^{-1}  \vert \ell \vert^\alpha \vert k - \ell \vert^\alpha  \Biggr|.\end{align*}
We then bound the right-hand side above by
\begin{multline*}
	  \left( \sup_{\ell \in \mathbb{Z}^3 \setminus \{0\}}  \frac{A_{ \ell,  \ell t}}{\vert \ell \vert^\alpha}   \vert \widehat{\phi}_l (t) \vert \right) \sup_{k \in \mathbb{Z}^3 \setminus \{0\}} A_{k, kt} \frac{1}{\vert k \vert^{ \alpha }}  \sum_{ \substack{ \ell \in \mathbb{Z}^3 \setminus \{0\} \\ k \neq \ell }} \frac{A_{k - \ell, (k - \ell)t}}{\vert k - \ell \vert^\alpha}  \vert  \widehat{\psi}_{k - \ell} (t) \vert A_{ \ell,  \ell t}^{-1}  A_{k - \ell, (k - \ell)t}^{-1}  \vert \ell \vert^\alpha \vert k - \ell \vert^\alpha \\
	   \leq  \left( \sup_{\ell \in \mathbb{Z}^3 \setminus \{0\}}  \frac{A_{ \ell,  \ell t}}{\vert \ell \vert^\alpha}   \vert \widehat{\phi}_l (t) \vert \right) \left( \sup_{\ell \in \mathbb{Z}^3 \setminus \{0\}}  \frac{A_{\ell , \ell t}}{\vert \ell \vert^\alpha}   \vert \widehat{\psi}_\ell (t) \vert \right) \\
	   \qquad \qquad \times \Biggl( \sup_{k \in \mathbb{Z}^3 \setminus \{0\}} A_{k, kt} \frac{1}{\vert k \vert^{ \alpha }}   \sum_{ \substack{ \ell \in \mathbb{Z}^3 \setminus \{0\} \\ k \neq \ell }}  A_{ \ell,  \ell t}^{-1}  A_{k - \ell, (k - \ell)t}^{-1}  \vert \ell \vert^\alpha \vert k - \ell \vert^\alpha \Biggr).
\end{multline*}
Note that, if we can prove that
\begin{equation}
\label{Claim_property_F}
 \sum_{ \substack{ \ell \in \mathbb{Z}^3 \setminus \{0\} \\ k \neq \ell }} A_{ \ell,  \ell t}^{-1}  A_{k - \ell, (k - \ell)t}^{-1}  \vert \ell \vert^\alpha \vert k - \ell \vert^\alpha \leq C A_{k,kt }^{-1}  \vert k \vert^\alpha,
\end{equation}
we can conclude as follows:
\begin{align*}
	F[\phi \psi](t,z) & \leq  \left( \sup_{\ell \in \mathbb{Z}^3 \setminus \{0\}}  \frac{A_{ \ell,  \ell t}}{\vert \ell \vert^\alpha}   \vert \widehat{\phi}_l (t) \vert \right)\left( \sup_{\ell \in \mathbb{Z}^3 \setminus \{0\}}  \frac{A_{\ell , \ell t}}{\vert \ell \vert^\alpha}   \vert \widehat{\psi}_\ell (t) \vert \right)  \sup_{k \in \mathbb{Z}^3 \setminus \{0\}} C A_{k, kt} \frac{1}{\vert k \vert^{ \alpha }}  A_{k,kt }^{-1}  \vert k \vert^\alpha \\
	&\leq C \left( \sup_{\ell \in \mathbb{Z}^3 \setminus \{0\}}  \frac{A_{ \ell,  \ell t}}{\vert \ell \vert^\alpha}   \vert \widehat{\phi}_l (t) \vert \right) \left( \sup_{\ell \in \mathbb{Z}^3 \setminus \{0\}}  \frac{A_{\ell , \ell t}}{\vert \ell \vert^\alpha}   \vert \widehat{\psi}_\ell (t) \vert \right) = C F[\phi] (t,z) F[\psi] (t,z).
\end{align*}
Let us now prove (\ref{Claim_property_F}). We want to show that
\begin{equation*}
\sum_{ \substack{ \ell \in \mathbb{Z}^3 \setminus \{0\} \\ k \neq \ell }} A_{k,kt } A_{ \ell,  \ell t}^{-1}  A_{k - \ell, (k - \ell)t}^{-1} \frac{\vert \ell \vert^{\alpha} \vert k - \ell \vert^{\alpha}}{\vert k \vert^{\alpha}} \leq C.
\end{equation*}
First we note that $e^{z \langle k, kt \rangle^{\gamma}} \leq e^{z \langle \ell,  \ell t \rangle^{\gamma}} e^{z \langle k - \ell, (k - \ell)t \rangle^{\gamma}}$, where we have used the following inequality for the Japanese brackets, $\langle k, kt \rangle  \leq \langle \ell,  \ell t \rangle + \langle k - \ell, (k - \ell)t \rangle$ and the sub-additivity of the map $x \rightarrow x^\gamma$.  Also, we claim that for $k \neq 0$ and $\ell \neq 0$, there exists a constant $C$ such that
\begin{align}
\label{inequaltiy frac}
	\frac{\vert \ell \vert^{\alpha} \vert k - \ell \vert^{ \alpha}}{\vert k \vert^{ \alpha}}  \leq  C \frac{\langle \ell,  \ell t \rangle^{ \alpha} \langle k - \ell, (k - \ell)t \rangle^{ \alpha}}{\langle k,kt \rangle^{ \alpha}}.
\end{align}
Then, for $\sigma - \alpha > 3$ we have
\begin{align*}
  \sum_{ \substack{ \ell \in \mathbb{Z}^3 \setminus \{0\} \\ k \neq \ell }} & \frac{e^{z \langle k, kt \rangle^{\gamma}}}{e^{z \langle \ell,  \ell t \rangle^{\gamma}} e^{z \langle k - \ell, (k - \ell)t \rangle^{\gamma}}} \frac{\langle k, kt \rangle^{\sigma} }{\langle \ell,  \ell t \rangle^{\sigma} \langle k - \ell, (k - \ell)t \rangle^{\sigma} } \frac{\vert \ell \vert^{ \alpha} \vert k - \ell \vert^{ \alpha}}{\vert k \vert^{ \alpha}} \\
& \leq C  \sum_{ \substack{ \ell \in \mathbb{Z}^3 \setminus \{0\} \\ k \neq \ell }} \frac{\langle k, kt \rangle^{\sigma - \alpha} }{\langle \ell,  \ell t \rangle^{\sigma - \alpha} \langle k - \ell, (k - \ell)t \rangle^{\sigma - \alpha} } .
\end{align*}
We now split the sum into three regions.  
\begin{itemize}
\item $\mathcal{A}= \{ \ell \in \mathbb{Z}^3 \setminus \{0\} \ \mbox{s.t. } \ \vert k - \ell \vert > 3 \vert \ell \vert$  \},
\item $ \mathcal{B}= \{ \ell \in \mathbb{Z}^3 \setminus \{0\} \ \mbox{s.t. } \ \vert k - \ell \vert < \frac{1}{3} \vert \ell \vert$  \},
\item $\mathcal{C} = \{ \ell \in \mathbb{Z}^3 \setminus \{0\} \ \mbox{s.t. } \ \frac{1}{3} \vert \ell \vert < \vert k - \ell \vert < 3 \vert \ell \vert$  \}.
\end{itemize}
In  region $\mathcal{A}$, we have $\vert k \vert > 2 \vert \ell \vert$, therefore $\vert k - \ell \vert > \frac{1}{2} \vert k \vert.$ In  region $\mathcal{B}$, we have $ \frac{2}{3} \vert \ell \vert <  \vert k \vert < \frac{4}{3} \vert \ell \vert$. In  region $\mathcal{C}$, we have $ \frac{1}{3} \vert \ell \vert < \vert k - \ell \vert$ and $\vert k \vert < 4 \vert \ell \vert.$ Therefore, using these inequalities on the different regions and $\sigma - \alpha > 3$, we obtain
\begin{multline*}
\sum_{\mathcal{A}}  \frac{\langle k, kt \rangle^{\sigma - \alpha} }{\langle \ell,  \ell t \rangle^{\sigma - \alpha} \langle k - \ell, (k - \ell)t \rangle^{\sigma - \alpha} } + \sum_{\mathcal{B}} \frac{\langle k, kt \rangle^{\sigma - \alpha} }{\langle \ell,  \ell t \rangle^{\sigma - \alpha} \langle k - \ell, (k - \ell)t \rangle^{\sigma - \alpha} } \\
+ \sum_{\mathcal{C}} \frac{\langle k, kt \rangle^{\sigma - \alpha} }{\langle \ell,  \ell t \rangle^{\sigma - \alpha} \langle k - \ell, (k - \ell)t \rangle^{\sigma - \alpha} }\\
 \leq C \Bigg(\sum_{\mathcal{A}} \frac{\langle k, kt \rangle^{\sigma - \alpha} }{\langle \ell,  \ell t \rangle^{\sigma - \alpha}  \langle k, kt \rangle^{\sigma - \alpha} } + \sum_{\mathcal{B}}  \frac{ \langle \ell,  \ell t \rangle^{\sigma - \alpha} }{\langle \ell,  \ell t \rangle^{\sigma - \alpha} \langle k - \ell, (k - \ell)t \rangle^{\sigma - \alpha} }\\
 + \sum_{\mathcal{C}} \frac{  \langle \ell,  \ell t \rangle^{\sigma - \alpha} }{\langle \ell,  \ell t \rangle^{\sigma - \alpha}  \langle \ell,  \ell t \rangle^{\sigma - \alpha} }\Bigg) \leq C.
\end{multline*}
Finally, let us prove (\ref{inequaltiy frac}).  We want to show 
\begin{align*}
	\frac{\vert \ell \vert \vert k - \ell \vert \langle k, kt \rangle}{\vert k \vert \langle \ell,  \ell t \rangle \langle k - \ell, (k - \ell)t \rangle} \leq C.
\end{align*}
We split the analysis in two different cases. 
\begin{itemize}
\item $t>1$: in this case we have 
\begin{equation*}
	\langle k, kt \rangle = (1 + \vert k \vert^2 + \vert kt \vert^2)^{\frac{1}{2}} \leq C \vert kt \vert.
\end{equation*}
Therefore,
\begin{align*}
	\frac{\vert \ell \vert \vert k - \ell \vert \langle k, kt \rangle}{\vert k \vert \langle \ell,  \ell t \rangle \langle k - \ell, (k - \ell)t \rangle} \leq C \frac{\vert \ell \vert \vert k - \ell \vert \vert kt \vert}{\vert k \vert \vert \ell t \vert \vert (k - \ell)t \vert} \leq C \frac{1}{t} \leq C,
\end{align*}
where we used $\vert \ell t \vert < \langle \ell,  \ell t \rangle$ and $\vert (k - \ell)t \vert < \langle k - \ell, (k - \ell)t \rangle$.
\item $t \leq 1$: in this case we have 
\begin{equation*}
	\langle k, kt \rangle = (1 + \vert k \vert^2 + \vert kt \vert^2)^{\frac{1}{2}} \leq C \vert k \vert.
\end{equation*}
Therefore,
\begin{align*}
	\frac{\vert \ell \vert \vert k - \ell \vert \langle k, kt \rangle}{\vert k \vert \langle \ell,  \ell t \rangle \langle k - \ell, (k - \ell)t \rangle} \leq C \frac{\vert \ell \vert \vert k - \ell \vert \vert k \vert}{\vert k \vert \vert \ell  \vert \vert k - \ell \vert} \leq C,
\end{align*}
where we used $\vert \ell \vert < \langle \ell,  \ell t \rangle$ and $\vert k - \ell \vert < \langle k - \ell, (k - \ell)t \rangle$.
\end{itemize}
Hence, we have our claim
\begin{align*}
	\frac{\vert \ell \vert^{\alpha} \vert k - \ell \vert^{ \alpha}}{\vert k \vert^{ \alpha}} \leq C \frac{\langle \ell,  \ell t \rangle^{ \alpha} \langle k - \ell, (k - \ell)t \rangle^{ \alpha}}{\langle k, kt \rangle^{ \alpha}}.
\end{align*}
\end{proof}

\begin{lem}
\label{property_F_2}
Let $F[\cdot]$ be the generator function defined by (\ref{GenF}) and $h$ be an analytic function with analyticity radius $R$,
\begin{equation*}
	h(z) = \sum_{n \geq 0} a_n z^n.
\end{equation*}
Define $\widetilde{h}(z) = \sum_{n \geq 0} \vert a_n \vert z^n$. Let $ \phi $ be a function such that $\left\Vert \phi \right\Vert_{L^\infty (\bt^3)} \leq R$. Then there exists a universal constant $C$ such that
\begin{equation}
	F[h(\phi)](t,z) \leq  \widetilde{h}\Big(C F[\phi](t,z)\Big).
\end{equation}

\end{lem}

\begin{proof} The proof is similar to \cite[Lemma 1.2]{EG_TN}.
By Lemma \ref{property_F} we know that $F[a_n \phi^n](t,z) \leq \vert a_n  \vert C^n F[\phi]^n (t,z)$, where the constant $C$ is the one from Lemma \ref{property_F}.  Therefore,
\begin{align*}
	F[h(\phi)](t,z) = F \Big[\sum_{n \geq 0} a_n \phi^n \Big](t,z) \leq \sum_{n \geq 0} \vert a_n \vert   C^n F[ \phi ]^n (t,z) = \widetilde{h}\Big(C F[\phi](t,z)\Big).
\end{align*}
\end{proof}

\begin{lem}
\label{lemma_F<G}
 Let $g(t,x,v) := f(t,x+vt,v)$ and 
 \begin{equation*}
	\rho (t,x) = \int_{\br^3} g(t, x-vt,v) \ dv.
\end{equation*}
 Let $F[\rho]$ and $G[g]$ be the generator functions defined in (\ref{GenF}) and (\ref{GenG}). Then there exists a constant $C,$ depending only on $\lambda_1,$ such that
\begin{equation*}
	F[\rho] (t,z) \leq C G[g(t)]^{\frac{1}{2}} (z).
\end{equation*}
for any $z\in[0,\lambda_1]$, $t \geq 0$.
\end{lem}

\begin{proof}
For this proof we will need the inequality (2.10) in \cite{BMM_13}. Namely, for $A_{k,\eta}=e^{z \langle k, \eta \rangle^{\gamma}}   \langle k, \eta \rangle^{\sigma}$ and  $j \in \bn^d$ a multi-index, there exists a constant $\bar C= \bar C(j,\lambda_1)$ such that
\begin{equation}
\label{derivative Aketa}
	\left| \partial_\eta^j A_{k, \eta} \right| \leq \bar C(j)\frac{1}{\langle k, \eta \rangle^{\vert j \vert (1-\gamma)} } A_{k, \eta} \leq \bar C(j) A_{k,\eta},
\end{equation}
for every $z\in [0,\lambda_1].$
Let us now prove the result. First, we can show that $\widehat{\rho}_k (t) = \widehat{g}_{k,kt} (t)$.  Indeed,
\begin{align*}
	\widehat{\rho}_{k} (t) &=    \int_{\bt^3} \rho(t,x) e^{-ikx} \ dx \\
	&=   \int_{\bt^3} \int_{\br^3} g(t,x-vt,v) \ dv \ e^{-ikx} \ dx  =  \int_{\bt^3 \times \br^3} g(t,y,v) e^{-iky} e^{-iktv} \ dy \ dv = \widehat{g}_{k,kt} (t).
\end{align*}
Therefore,
\begin{align*}
	A_{k,kt} \vert \widehat{\rho}_{k} (t)  \vert \frac{1}{\vert k \vert^\alpha} & \leq A_{k,kt} \vert \widehat{g}_{k,kt} (t) \vert \leq \sup_{\eta \in \br^3} A_{k,\eta} \vert \widehat{g}_{k,\eta} (t) \vert \leq \left\Vert A_{k,\eta}  \widehat{g}_{k,\eta} (t) \right\Vert_{H_\eta^3},
\end{align*}
where we used the $L^\infty$ Sobolev embedding.  Then by definition of $H_\eta^3$ and inequality (\ref{derivative Aketa}) we obtain,
\begin{align*}
	A_{k,kt} \vert \widehat{\rho}_{k} (t)  \vert \frac{1}{\vert k \vert^\alpha} & \leq \Biggl( \sum_{\vert j \vert \leq 3} \left\Vert \partial_\eta^j \big( A_{k,\eta}  \widehat{g}_{k,\eta} (t) \big)\right\Vert_{L_\eta^2}^2  \Biggr)^\frac{1}{2}  \leq  \Bigg( \sum_{\vert j \vert \leq 3} \sum_{i \leq j} \left( \frac{j!}{i! (j-i)!} \right)^2 \left\Vert \left( \partial_\eta^i A_{k,\eta} \right) \left( \partial_\eta^{(j-i)}  \widehat{g}_{k,\eta} (t) \right) \right\Vert_{L_\eta^2}^2  \Bigg)^\frac{1}{2} \\
		& \leq  \Bigg( \sum_{\vert j \vert \leq 3} \sum_{i \leq j}\bar C(i,\lambda_1) \frac{j!}{i! (j-i)!} \left\Vert  A_{k,\eta}  \left( \partial_\eta^{(j-i)}  \widehat{g}_{k,\eta} (t) \right) \right\Vert_{L_\eta^2}^2  \Bigg)^\frac{1}{2}  \\
		&\leq C(\lambda_1) \Bigg( \sum_{\vert j \vert \leq 3}  \left\Vert  A_{k,\eta}   \partial_\eta^j  \widehat{g}_{k,\eta} (t) \right\Vert_{L_\eta^2}^2  \Bigg)^\frac{1}{2} \\
	& \leq C\Bigg( \sum_{\vert j \vert \leq 3} \sum_{m \in \mathbb{Z}^3} \int_{\br^3}  A_{m,\eta}^2   \left| \partial_\eta^j  \widehat{g}_{m,\eta} (t) \right|^2 \ ds  \Bigg)^\frac{1}{2}  \leq C \,G[g(t)]^\frac{1}{2}(z).
\end{align*}
Hence,
\begin{equation*}
	F[\rho](t,z) = \sup_{k \in \mathbb{Z}^3 \setminus \{0\}} A_{k,kt} \vert \widehat{\rho}_{k} (t)  \vert \frac{1}{\vert k \vert^\alpha}  \leq  C G[g(t)]^\frac{1}{2}(z).
\end{equation*}

\end{proof}

\begin{rem}
In particular, we note that 
\begin{equation}
\label{inequaltiy_g}
	\sup_{\eta \in \br^d} A_{k,\eta} \vert \widehat{g}_{k,\eta} (t) \vert \leq C G[g(t)]^\frac{1}{2}(z).
\end{equation}
We will need this inequality later in the proof of the Theorem \ref{Thm1}.
\end{rem}
Finally, we give two results on the initial condition. 

\begin{lem}
\label{Assumptions on the initial datum}
Let $f^0$ be the initial condition of (\ref{nlpvpme}) such that for $\varepsilon$ sufficiently small we have
\begin{equation*}
	G[f^0](\lambda_1) \leq \varepsilon.
\end{equation*}
Let $U$ be the potential in (\ref{nlpvpme}), i.e. $U$ solves
\begin{equation}
\label{PDE_U}
	-\Delta U + \beta U + h(U)=  \rho.
\end{equation}
Then,  there exists a constant $C_1$ such that 
\begin{equation*}
	F[\Delta U](0,\lambda_1) \leq C_1 \sqrt{\varepsilon}.
\end{equation*}

\end{lem}
\begin{proof}
Taking the Fourier transform of (\ref{PDE_U}), we get
\begin{equation*}
	\widehat{U}_k = \frac{1}{\beta + \vert k \vert^2}  \widehat{\rho}_{k} - \frac{1}{\beta + \vert k \vert^2} \widehat{h(U)}_k.
\end{equation*}
We then multiply both sides by $\vert k \vert^2$,
\begin{equation*}
	\vert k \vert^2 \widehat{U}_k = \frac{\vert k \vert^2}{\beta + \vert k \vert^2}  \widehat{\rho}_{k} - \frac{\vert k \vert^2}{\beta + \vert k \vert^2} \widehat{h(U)}_k.
\end{equation*}
 Therefore, by definition (\ref{GenF}), we get
 \begin{align*}
 	F[\Delta U](t,z)  \leq F[\rho](t,z) + F[h(U)](t,z) & \leq F[\rho](t,z) +\widetilde{h} \left( C F[U](t,z) \right) \\
 	& \leq F[\rho](t,z) + \widetilde{h} \left( C F[\Delta U](t,z) \right),
 \end{align*}
  where we used Lemma \ref{property_F_2}, and the fact that $\widetilde{h}$ is an increasing function with $F[U](t,z) \leq  F[\Delta U](t,z)$. Then, since the generator function is increasing in $z$, we have at $t=0$ and for every $z \in [0, \lambda_1]$, 
  \begin{equation}
  \label{Bound_F_Delta_U}
  	F[\Delta U](0,z)\leq F[\rho](0,\lambda_1) + \widetilde{h} \left( C F[\Delta U](0,z) \right).
  \end{equation}
Therefore, using Lemma \ref{lemma_F<G} and the assumption (\ref{IC}) on the initial data, we have
\begin{equation}
\label{Bound_F_rho_lambda1}
	F[\rho](0,\lambda_1) \leq C G[f^0]^\frac{1}{2}(\lambda_1) \leq C \sqrt{\varepsilon}.
\end{equation}
Thus,  inserting (\ref{Bound_F_rho_lambda1}) in (\ref{Bound_F_Delta_U}), we obtain
 \begin{align*}
 	F[\Delta U](0,z) \leq C \sqrt{\varepsilon}  + \widetilde{h}\left( C F[\Delta U](0,z) \right).
 \end{align*}
We now claim the following: there exists a small $\varepsilon_1 > 0$ such that,
\begin{equation}
\label{Claim_F_Delta_U}
	F[\Delta U](0,0) \leq \varepsilon_1.
\end{equation}
Before proving this claim, we first show how using it we can conclude the proof of the lemma.

For every $z \in [0, \lambda_1]$
\begin{eqnarray*}
    \left \{
    \begin{array}{l}
   F[\Delta U](0,z) \leq C \sqrt{\varepsilon}  + \widetilde{h}\left(C F[\Delta U](0,z) \right), \\
    F[\Delta U](0,0) \leq \varepsilon_1.
    \end{array}
    \right.
\end{eqnarray*}
Since $F[\Delta U](0,0) \leq \varepsilon_1 \ll 1$,  given $\delta > \varepsilon_1$ small, by continuity there exists a $z > 0$ such that $F[\Delta U](0,z) \leq  \delta < 1 $,  where $\delta$ will be fixed later. Then by assumption, $\widetilde{h}(x) = \mathcal{O}(x^2)$ as $x$ tends to zero,  we have 
 \begin{equation*}
	\widetilde{h}\Big(C F[\Delta U](0,z) \Big) \leq C (F[\Delta U] (0,z))^2  \leq C_0 \delta F[\Delta U](0,z).
\end{equation*}
Thus, $F[\Delta U](0,z) \leq C  \sqrt{\varepsilon} + C_0 \delta F[\Delta U](0,z)$. Hence, if we choose $\delta =  \frac{1}{2C_0} $ we get 
$$
F[\Delta U](0,z) \leq C \sqrt{\varepsilon} + \frac{1}{2} F[\Delta U](0,z)
$$
which implies
 \begin{equation*}
	F[\Delta U](0,z)\leq  C_1  \sqrt{\varepsilon},
\end{equation*}
for $z$ small.
Let us assume that $z^* < \lambda_1$ is the first value for which $F[\Delta U](0,z^*) = \delta.$ Then, if $\sqrt{\varepsilon} < \frac{\delta}{C_1}$, we have 
\begin{equation*}
	 \delta = F[\Delta U](0,z^*) \leq  C_1 \sqrt{\varepsilon} < \frac{ C_1 \delta}{C_1} = \delta,
\end{equation*} so we have a contradiction about  $z^* < \lambda_1$, hence $z^* = \lambda_1$.
Therefore, one can repeat the argument above for every $z \leq \lambda_1$, and we get:
$$
F[\Delta U](0,z) \leq C_1  \sqrt{\varepsilon},
$$
and in particular
\begin{equation*}
	F[\Delta U](0,\lambda_1) \leq C_1  \sqrt{\varepsilon}.
\end{equation*}
Let us prove Claim (\ref{Claim_F_Delta_U}). By (\ref{GenF}), we have
 \begin{align*}
 	F[\Delta U](0,0) &  =  \sup_{k \in \mathbb{Z}^3 \setminus \{0\}} \vert k \vert^2 \vert \widehat{ U}_k (0) \vert \langle k, 0 \rangle^{\sigma} \frac{1}{\vert k \vert^{\alpha }}  \\
	& \leq \bigg( \sum_{k \in \mathbb{Z}^3 \setminus \{0\}} \vert k \vert^4 \vert  \widehat{ U}_k (0) \vert^2 \langle k \rangle^{2\sigma}  \bigg)^{\frac{1}{2}} 
 	 \leq \bigg( \sum_{k \in \mathbb{Z}^3}  \vert \widehat{ U}_k (0) \vert^2 \langle k \rangle^{{2\sigma+4}  \bigg)^{\frac{1}{2}}}
 	\leq  \| U(0) \|_{H^{ \sigma + 2} (\mathbb{T}^d)}.
 \end{align*} 
 
Then, we can show that the $H^{\sigma + 2}$-norm of $\rho$ is bounded in the following way:
 \begin{align*}
 	\| \rho(0) \|_{H^{\sigma + 2} (\mathbb{T}^d)} & = \bigg( \sum_{k \in \mathbb{Z}^3 \setminus \{0\}}  \vert \widehat{ \rho}_k (0) \vert^2 \langle k \rangle^{2\sigma + 4} \bigg)^{\frac{1}{2}}  \leq  \bigg( \sum_{k \in \mathbb{Z}^3 \setminus \{0\}} e^{2 \lambda_1 \langle k \rangle^{\gamma}} \vert \widehat{\rho}_k (0) \vert^2 \langle k \rangle^{2\sigma} \frac{\vert k \vert^{2\alpha}}{\vert k \vert^{2\alpha}} e^{- 2 \lambda_1 \langle k \rangle^{\gamma}} \langle k \rangle^{4} \bigg)^{\frac{1}{2}} \\
 	& \leq  \bigg( \sup_{k \in \mathbb{Z}^3 \setminus \{0\}} e^{ \lambda_1 \langle k \rangle^{\gamma}} \vert \widehat{\rho}_k (0) \vert \langle k \rangle^{\sigma} \frac{1}{\vert k \vert^\alpha} \bigg) \bigg( \sum_{k \in \mathbb{Z}^3 \setminus \{0\}}  e^{- 2 \lambda_1 \langle k \rangle^{\gamma}} \langle k \rangle^{4+ 2 \alpha} \bigg)^{\frac{1}{2}} \\
 	& \leq C F[\rho](0,\lambda_1) \leq C \sqrt{\varepsilon},
 \end{align*}
 where we used (\ref{Bound_F_rho_lambda1}) for the last inequality.  Therefore, looking at the equation (\ref{PDE_U}) at time $t=0$,  we have by elliptic regularity and by continuous dependence on the right-hand side,  that there exists an $\varepsilon_1>0$ depending on $\varepsilon$ such that
 \begin{equation*}
 	\| U(0) \|_{H^{\sigma + 2} (\mathbb{T}^d)} \leq \varepsilon_1.
 \end{equation*}
 This proves (\ref{Claim_F_Delta_U}) and ends the proof of Lemma \ref{Assumptions on the initial datum}.
 \end{proof}
\begin{lem}
\label{Assumptions on the initial datum 2}
	If $F[\Delta U](0,\lambda_1) \leq C_1 \sqrt{\varepsilon}$ then $F[ U](0,z) \leq C_1 \sqrt{\varepsilon}$ for every $z \in [0,\lambda_1].$
\end{lem}
\begin{proof}
It suffices to write the definition of generator functions (\ref{GenF}),
\begin{align*}
	F[U](t,z) &= \sup_{k \in \mathbb{Z}^3 \setminus \{0\}} e^{z \langle k, kt \rangle^{\gamma}} \vert \widehat{U}_k (t) \vert \langle k, kt \rangle^{\sigma} \frac{1}{\vert k \vert^{ \alpha }}\\
	& \leq \sup_{k \in \mathbb{Z}^3 \setminus \{0\}} e^{z \langle k, kt \rangle^{\gamma}} \vert k \vert^2 \vert \widehat{U}_k (t) \vert \langle k, kt \rangle^{\sigma} \frac{1}{\vert k \vert^{ \alpha }} = F[\Delta U](t,z).
\end{align*}

\end{proof}

\section{Linear Landau damping}
\label{section Linear Landau damping}
In this section, we prove the linear Landau damping of the system (\ref{lpvpme}). In Section \ref{Equation on the density}, we derive a closed equation for the density $\rho$.  Next, we obtain estimates on the resolvent kernel $K$ and $\rho$ in Section \ref{Resolvent estimates} and Section \ref{Pointwise estimates}. Then using the generator functions, we bound $\rho$ and $U$ in Section \ref{Estimate on F_rho} and Section \ref{Estimate on F_U}. Finally, Section \ref{Asymptotic behaviour} gives us the asymptotic behaviour of a solution of the linear dynamic.

For simplicity, we introduce $g(t,x,v) := f(t,x+vt,v)$,  the pull-back of $f$ by the free transport flow. And we observe that 
\begin{equation*}
    \partial_t g = \partial_t f + v \cdot \nabla_x f, \quad \nabla_x g = \nabla_x f.
\end{equation*}
Therefore the Vlasov equation (\ref{lpvpme}) for $g$ becomes
\begin{equation}
\label{equation_g} 
    \partial_t g + E(t,x+vt) \cdot \nabla_v \mu = 0.
\end{equation}

\subsection{Equation for the density}
\label{Equation on the density}

Throughout this section, we focus on the case $k\neq0$. Indeed, by the conservation of the mass and the fact that the mass of the initial data $f^0$ is zero, we have
\begin{equation*}
	\widehat{\rho}_0(t) = \int_{\mathbb{T}^3} \rho(t,x) \ dx = \int_{\mathbb{T}^3 \times \mathbb{R}^3} f(t,x,v) \ dx \ dv = \int_{\mathbb{T}^3 \times \mathbb{R}^3} f^0(x,v) \ dx \ dv =  0.\\
\end{equation*}
Moreover, we don't state any global well-posedness for the equation  (\ref{equation_g}), but we can see Section \ref{section Linear Landau damping} as priori estimates on the solution $g$ with analytic regularity and for small initial data. And once we have proved that the estimates hold for all times, we have the global well-posedness of the system (\ref{lpvpme}). 

\begin{lem}
\label{lemma_rhohat}
Let g be the unique solution to the problem $(\ref{equation_g})$ with initial condition $g(0,x,v) = f^0(x,v)$. Then we have the following equation for the Fourier transform of $\rho$. 
\begin{equation}
\label{rhohat}
	\widehat{\rho}_k (t) + \frac{\vert k \vert^2}{\beta + \vert k \vert^2}\int_0^t  \widehat{\rho}_k(s)  (t-s)  \widehat{\mu}(k(t-s)) \ ds = \widehat{S}_k (t),
\end{equation}
where the source term is given by
 \begin{equation}
 \label{def_S_k}
 	 \widehat{S}_k (t) =  \widehat{f^0}_{k,kt} + \frac{\vert k \vert^2}{\beta + \vert k \vert^2} \int_0^t  \widehat{h(U)}_k(s) (t-s) \widehat{ \mu } (k (t - s)) \ ds. 
 \end{equation}
\end{lem}

\begin{proof}

We start by taking the Fourier transform with respect to $k$ and $\eta$ of (\ref{equation_g}). So using the properties of the Fourier transform we have,
\begin{equation*}
	\widehat{\partial_t g }(t, k, \eta) = \partial_t \widehat{g}_{k, \eta} (t),
\end{equation*}
and 
\begin{equation}
\label{E_mu_big_hat}
	\reallywidehat{\left( E(t,x+vt) \cdot \nabla_v \mu \right) }(k,\eta) = \widehat{E}_k (t) e^{ikvt} \cdot \widehat{\nabla_v \mu } (\eta) =  \widehat{E}_k (t) \cdot  \widehat{\nabla_v \mu } (\eta - kt).
\end{equation}
Therefore, (\ref{equation_g}) transforms into 
\begin{equation*}
	\partial_t \widehat{g}_{k, \eta} (t) + \widehat{E}_k (t) \cdot  \widehat{\nabla_v \mu } (\eta - kt) = 0.
\end{equation*}
The solution of this differential equation is given by 
\begin{equation}
\label{equ_g_hat}
	\widehat{g}_{k,\eta} (t) = \widehat{g}_{k,\eta} (0) - \int_0^t \widehat{E}_k (s) \cdot \widehat{\nabla_v \mu } (\eta - ks) \ ds.
\end{equation}
Next, we consider the Fourier mode $\eta = kt$. Recall that $\widehat{g}_{k,kt} (t) = \widehat{\rho}_k (t) $ and $\widehat{g}_{k,kt} (0) = \widehat{f^0}_{k,kt}.$ Hence, we have 
\begin{equation}
\label{rho_hat_aux1}
	\widehat{\rho}_{k} (t)  = \widehat{f^0}_{k,kt} - \int_0^t \widehat{E}_k (s) \cdot  \widehat{\nabla_v \mu } (k (t - s)) \ ds.
\end{equation}
Noticing that
\begin{align}
\label{E_mu}
	\widehat{E}_k (s) \cdot  \widehat{\nabla_v \mu } (k(t - s)) & = \widehat{E}_k (s) \cdot  ik(t-s) \widehat{ \mu } (k(t - s)).
	\end{align}
(\ref{rho_hat_aux1}) becomes
\begin{equation}
\label{rho_hat_aux2}
	\widehat{\rho}_{k} (t)  =  \widehat{f^0}_{k,kt} - \int_0^t ik \cdot  \widehat{E}_k (s)  (t-s) \widehat{ \mu } (k(t - s))\ ds.
\end{equation}
Taking the Fourier transform of the force field and the Poisson coupling of (\ref{lpvpme}), we have
\begin{equation}
\label{E_hat}
	\widehat{E}_k = - ik \widehat{U}_k,  \quad \mbox{or equivalently} \quad ik \cdot \widehat{E}_k = \vert k \vert^2 \widehat{U}_k,
\end{equation}
and
\begin{equation}
\label{rho_hat}
	 - \widehat{\Delta U}_k +\beta  \widehat{U}_k +\widehat{h(U)}_k= (\vert k \vert^2 + \beta) \widehat{U}_k +   \widehat{h(U)}_k =  \widehat{\rho}_{k}.
\end{equation}
The latter equation can be rewritten as follows:
\begin{equation*}
	\widehat{U}_k = \frac{1}{\beta + \vert k \vert^2}  \widehat{\rho}_{k} - \frac{1}{\beta + \vert k \vert^2} \widehat{h(U)}_k.
\end{equation*}
Putting (\ref{E_hat}) and (\ref{rho_hat}) together, we get
\begin{equation}
\label{E_hat2}
	 ik \cdot \widehat{E}_k = \frac{\vert k \vert^2}{\beta + \vert k \vert^2}  \widehat{\rho}_{k} - \frac{\vert k \vert^2}{\beta + \vert k \vert^2} \widehat{h(U)}_k.
\end{equation}
Hence from (\ref{rho_hat_aux2}), we have the desired result:
\begin{equation*}
	\widehat{\rho}_{k} (t) + \frac{\vert k \vert^2}{\beta + \vert k \vert^2} \int_0^t \widehat{\rho}_{k}(s) (t-s) \widehat{ \mu } (k (t - s)) \ ds  = \widehat{S}_k (t),
\end{equation*}
where
 \begin{equation*}
 	 \widehat{S}_k (t) =  \widehat{f^0}_{k,kt} + \frac{\vert k \vert^2}{\beta + \vert k \vert^2} \int_0^t  \widehat{h(U)}_k(s) (t-s) \widehat{ \mu } (k (t - s)) \ ds. 
 \end{equation*}
\end{proof}

\subsection{Resolvent estimates}
\label{Resolvent estimates}

In this section, we consider the Laplace transform of (\ref{rhohat}) to transform an integral equation on $\rho$ into an algebraic equation. First, let us recall that the Laplace transform is linear and 
\begin{equation}\label{Property-Laplace}
	\mathcal{L}\left[\int_0^t \phi(s) \psi(t-s) \ ds \right](\lambda) = \mathcal{L}[\phi](\lambda) \mathcal{L}[\psi](\lambda).
\end{equation}
Thus, we get
\begin{equation*}
	\mathcal{L}[\widehat{\rho}_k](\lambda) + \frac{\vert k \vert^2}{\beta + \vert k \vert^2}  \mathcal{L}[t\widehat{\mu}(kt)](\lambda) \mathcal{L}[\widehat{\rho}_k](\lambda) = \mathcal{L}[\widehat{S}_k](\lambda),
\end{equation*}
which can be rewritten as
\begin{equation}
\label{Lrho1}
	\mathcal{L}[\widehat{\rho}_k](\lambda)  =  \frac{\mathcal{L}[\widehat{S}_k](\lambda)}{1 +  \frac{\vert k \vert^2}{\beta + \vert k \vert^2}  \mathcal{L}[t\widehat{\mu}(kt)](\lambda) }.
\end{equation}
The stability condition (\ref{stability_cond}) ensures that the denominator never vanishes. 

\begin{rem}
We note that condition (\ref{stability_cond}), without the factor $\frac{\vert k \vert^2}{\beta + \vert k \vert^2}$, is the usual stability condition for homogeneous equilibrium of the (VP) system (see \cite{BMM_13, GNR, CM_CV}).  The extra factor $\frac{\vert k \vert^2}{\beta + \vert k \vert^2}$ comes from the different coupling used in system (\ref{nlpvpme}) compared to (VP).  Our stability condition (\ref{stability_cond}) is similar to the one used for the screened (VP) (see for example \cite{BMM_16} for the case $\beta > 0$ or \cite{HKNR,  huang_sharp_2022} for the case $\beta = 1$). Moreover, it is possible to show that the Penrose stability condition \cite{Penrose} implies our  condition  (\ref{stability_cond}) using the same argument as in \cite[Section 4.3]{BMM_13} (see also \cite[Section 3.4]{Villani_notes}).
\end{rem}

We define the resolvent kernel, $\widetilde{K}_k (\lambda)$, as
\begin{equation}
\label{resolvent}
	\widetilde{K}_k (\lambda) := - \frac{ \frac{\vert k \vert^2}{\beta + \vert k \vert^2} \mathcal{L}[t \widehat{\mu} (kt)](\lambda)}{1+\frac{\vert k \vert^2}{\beta + \vert k \vert^2} \mathcal{L}[t \widehat{\mu} (kt)](\lambda)}.
\end{equation}
Therefore, (\ref{Lrho1}) becomes
\begin{equation}
\label{Lrho2}
	\mathcal{L}[\widehat{\rho}_k](\lambda)  = \mathcal{L}[\widehat{S}_k](\lambda) +  \widetilde{K}_k (\lambda) \mathcal{L}[\widehat{S}_k](\lambda).
\end{equation}
Thanks to \cite[Lemma 3.2] {GNR}, we have the following property of the resolvent kernel.
\begin{lem}
\label{lemmaresolvent}
Let $\theta_0$ be the constant in (\ref{Pmu1}). Then there is a positive constant $\theta_1 < \frac{1}{2}\theta_0$ so that the function $\widetilde{K}_k (\lambda)$ is an analytic function in $ \lbrace \Re \lambda \geq - \theta_1 \vert k \vert \rbrace $. Moreover, there exists a constant $C$ such that
\begin{equation*}
	\vert \widetilde{K}_k (\lambda) \vert \leq \frac{C}{1 + \vert k \vert^2 +\vert \Im \lambda \vert^2},
\end{equation*}
uniformly in $\lambda$ and $k \neq 0$ so that $\Re \lambda = - \theta_1 \vert k \vert$.

\end{lem}

\begin{proof}
We can do the same proof as in \cite[Lemma 3.2] {GNR} using the bound (recall that $\beta\ge0$),
\begin{equation*}
	\left| \frac{\vert k \vert^2}{\beta + \vert k \vert^2} \mathcal{L}[t \widehat{\mu} (kt)](\lambda) \right| \leq \left| \mathcal{L}[t \widehat{\mu} (kt)](\lambda) \right|.
\end{equation*}

\end{proof}

\subsection{Inverse Laplace transform of the resolvent kernel}
\label{Pointwise estimates}

Let us define the inverse Laplace transform of the resolvent kernel $\widetilde{K}_k (\lambda)$ as:
\begin{equation}
\label{K_hat}
	\widehat{K}_k (t) = \frac{1}{2 \pi i} \int_{\lbrace \Re \lambda= \gamma_0  \rbrace} \widetilde{K}_k (\lambda) e^{\lambda t} \ d\lambda,
\end{equation}
for some large constant $\gamma_0 > 0$.
Then, we have the following property:

\begin{lem}\cite[Proposition 3.3]{GNR}
	The inverse Laplace transform of the resolvent kernel $\widetilde{K}_k (\lambda)$ satisfies 
\begin{equation}
\label{inverse_resolvent}
	\vert \widehat{K}_k (t) \vert \leq C e^{- \theta_1 \vert kt \vert},
\end{equation}
where the constant $\theta_1$ comes from Lemma \ref{lemmaresolvent}.
\end{lem}
Next, taking the inverse Laplace transform of (\ref{Lrho2}) we have the following proposition on the density,

\begin{prop}
\label{prop_rhohat3}
Let $\widehat{\rho}_k (t)$ be the unique solution of the  problem (\ref{rhohat}), then we can write $\widehat{\rho}_k (t)$ as
\begin{equation}
\label{rhohat3}
	\widehat{\rho}_{k} (t)  = \widehat{f^0}_{k,kt} +  \int_0^t \widehat{K}_k (t-s) \widehat{f^0}_{k,ks} \ ds - \int_0^t \widehat{K}_k (t-s) \widehat{h(U)}_k(s) \ ds,
\end{equation}
with $ \widehat{K}_k (t)$ satisfying (\ref{inverse_resolvent}).
\end{prop}

\begin{rem}
 If $f^0$ is analytic or Gevrey, the first two terms on the right-hand side of (\ref{rhohat3}) decay exponentially fast when $t$ goes to infinity. Unfortunately, without additional information on $\widehat{h(U)}_k$, we only know that the third term is bounded by some constant.
 \end{rem} 
\begin{proof}
We take the inverse Laplace transform of (\ref{Lrho2}), and we recall the property of the Laplace transform \eqref{Property-Laplace}. Then we obtain 
\begin{equation}
\label{rhohat2}
	\widehat{\rho}_k (t) = \widehat{S}_k (t) + \int_0^t \widehat{K}_k (t-s) \widehat{S}_k (s) \ ds,
\end{equation}
where $\widehat{S}_k (t)$ is given by (\ref{def_S_k}) and $\widehat{K}_k(t)$ is given by (\ref{K_hat}).  Finally, we need to show,
\begin{multline*}
	 \frac{\vert k \vert^2}{\beta + \vert k \vert^2} \int_0^t  \widehat{h(U)}_k(s) (t-s) \widehat{ \mu } (k (t - s)) \ ds\\ +  \int_0^t \widehat{K}_k (t-s)  \frac{\vert k \vert^2}{\beta + \vert k \vert^2} \int_0^s \widehat{h(U)}_k(s') (s-s') \widehat{ \mu } (k (s - s')) \ ds' \ ds \\
	 = - \int_0^t \widehat{K}_k (t-s) \widehat{h(U)}_k(s) \ ds.
\end{multline*}
Let us write $C_{k,\beta} =  \frac{\vert k \vert^2}{\beta + \vert k \vert^2}$, $\phi_1(s) = \widehat{h(U)}_k(s)$ and $\phi_2(s) = s \widehat{ \mu } (k s)$.
We want to show that
\begin{align*}
	 C_{k,\beta}  \int_0^t  \phi_1(s)  \phi_2(t - s) \ ds +  C_{k,\beta} \int_0^t \widehat{K}_k (t-s)  \int_0^s \phi_1(s') \phi_2(s - s') \ ds' \ ds = - \int_0^t \widehat{K}_k (t-s) \phi_1(s) \ ds.
\end{align*}
Taking the Laplace transform of the left-hand side, and using by the definition of resolvent kernel (\ref{resolvent}), we get
\begin{multline*}
	C_{k,\beta}  \mathcal{L}[\phi_1](\lambda) \mathcal{L}[\phi_2](\lambda) + C_{k,\beta}  \widetilde{K}_k (\lambda) \mathcal{L}[\phi_1](\lambda) \mathcal{L}[\phi_2](\lambda) \\= C_{k,\beta}  \mathcal{L}[\phi_1](\lambda) \mathcal{L}[\phi_2](\lambda) - C_{k,\beta}  \frac{C_{k,\beta}  \mathcal{L}[\phi_2](\lambda)}{1+C_{k,\beta} \mathcal{L}[\phi_2](\lambda)} \mathcal{L}[\phi_1](\lambda) \mathcal{L}[\phi_2](\lambda) \\
	= \frac{C_{k,\beta}  \mathcal{L}[\phi_1](\lambda) \mathcal{L}[\phi_2](\lambda)}{1+C_{k,\beta} \mathcal{L}[\phi_2](\lambda)} = -   \widetilde{K}_k (\lambda) \mathcal{L}[\phi_1](\lambda).
\end{multline*}
Hence, applying the inverse Laplace transform we obtain
\begin{align*}
	 C_{k,\beta}  \int_0^t  \phi_1(s)  \phi_2(t - s) \ ds  + C_{k,\beta} \int_0^t \widehat{K}_k (t-s)  \int_0^s \phi_1(s') \phi_2(s - s') \ ds' \ ds & = \mathcal{L}^{-1} \Big[ -   \widetilde{K}_k (\lambda) \mathcal{L}[\phi_1](\lambda) \Big] \\
	 & =  - \int_0^t \widehat{K}_k (t-s) \phi_1(s) \ ds.
\end{align*}
\end{proof}

\subsection{Estimate on the norm \texorpdfstring{$F[\rho](t,z)$}{Lg}}
\label{Estimate on F_rho}

We want to obtain a uniform bound on $F[\rho](t,z)$ to show the decay of the Fourier transform of the density $\rho$ as  $t$ goes to infinity.

\begin{prop}
\label{Prop_estimate_F_rho}
Assume that $\rho$ satisfies equation (\ref{rhohat3}). Then for every time $t \geq 0$ and $z \in [0, \frac{\theta_1}{2}] \cap [0,  \lambda_1]$ with $\theta_1$ defined in Lemma \ref{lemmaresolvent} and $\lambda_1$ defined in Theorem \ref{Thm1}, there holds for some constant $C>0$,
\begin{equation}
\label{estimate_F_rho}
	F[\rho](t,z) \leq  C \sqrt{\varepsilon} + C \sup_{s\leq t} \widetilde{h}\Big(C F[U](s,z)\Big).
\end{equation}
\end{prop}

\begin{proof}
First, recall the notation of the weight, $A_{k,kt}=e^{z \langle k, kt \rangle^{\gamma}}   \langle k, kt \rangle^{\sigma}$, in the generator function (\ref{GenF}). 
Then we proceed as follows: we multiply (\ref{rhohat3}) by $A_{k,kt}$, we take the absolute value, we apply the triangle inequality and we use the decay estimate (\ref{inverse_resolvent}). Then, we end up with
\begin{align*}
	A_{k,kt} \vert \widehat{\rho}_{k} (t)  \vert   \leq  A_{k,kt}  \vert \widehat{f^0}_{k,kt} \vert + C\,A_{k,kt}   \int_0^t e^{-\theta_1 \vert k \vert (t-s)} \vert \widehat{f^0}_{k,ks} \vert \ ds + C\,A_{k,kt}  \int_0^t e^{-\theta_1 \vert k \vert (t-s)} \vert\widehat{h(U)}_k(s) \vert \ ds.
\end{align*}
Next, we show that 
\begin{equation}\label{eq:est_theta_1}
	 A_{k,kt}   \int_0^t e^{-\theta_1 \vert k \vert (t-s)} \vert \widehat{f^0}_{k,ks} \vert \ ds  \leq C\int_0^t  e^{-\frac{\theta_1}{4} (t-s)} A_{k,ks} \vert \widehat{f^0}_{k,ks} \vert \ ds.
\end{equation}
We write the left-hand side as 
\begin{multline*}
	e^{z \langle k, kt \rangle^{\gamma}} \langle k, kt \rangle^{\sigma}  \int_0^t e^{-\theta_1 \vert k \vert (t-s)} \vert \widehat{f^0}_{k,ks} \vert \ ds \\
	=  \int_0^t e^{-\frac{1}{4} \theta_1 \vert k \vert (t-s)} \underbrace{e^{z \langle k, kt \rangle^{\gamma}} e^{-\frac{1}{2} \theta_1 \vert k \vert (t-s)}}_{=:\mathcal{A}}  \vert \widehat{f^0}_{k,ks} \vert \underbrace{ \langle k, kt \rangle^{\sigma} e^{-\frac{1}{4} \theta_1 \vert k \vert (t-s)}}_{=:\mathcal{B}} \ ds.
\end{multline*}
If we show that
\begin{equation*}
	\mathcal{A}:=e^{z \langle k, kt \rangle^{\gamma}} e^{-\frac{1}{2} \theta_1 \vert k \vert (t-s)} \leq e^{z \langle k, ks \rangle^{\gamma}},
\end{equation*}
and
\begin{equation*}
	\mathcal{B} := \langle k, kt \rangle^{\sigma} e^{-\frac{1}{4} \theta_1 \vert k \vert (t-s)}  \leq C \langle k, ks \rangle^{\sigma},
\end{equation*}
then we obtain \eqref{eq:est_theta_1}.
We begin with $\mathcal{A}.$ Applying the triangle inequality, $\langle k, \eta \rangle  \leq \langle k', \eta' \rangle + \langle k- k', \eta- \eta' \rangle $, we have
\begin{equation*}
	z \langle k, kt \rangle^{\gamma} - \frac{1}{2} \theta_1 \vert k \vert (t-s)  \leq z \langle k, ks \rangle^{\gamma} + z\langle 0, kt - ks \rangle^{\gamma}  - \frac{1}{2}\theta_1 \vert k \vert (t-s) \leq z \langle k, ks \rangle^{\gamma},
\end{equation*}
because $z\in [0, \frac{1}{2}\theta_1]$ and $\gamma \in (0, 1]$.
Therefore, 
\begin{equation*}
	\mathcal{A} :=e^{z \langle k, kt \rangle^{\gamma}} e^{-\frac{1}{2} \theta_1 \vert k \vert (t-s)} \leq e^{z \langle k, ks \rangle^{\gamma}}.
\end{equation*}
To evaluate $\mathcal{B}$, using $\langle k, \eta \rangle  \leq 2 \langle k', \eta' \rangle  \langle k- k', \eta- \eta' \rangle$ and the inequality $e^{- \theta x} \leq C \theta^{-\sigma} \langle x \rangle^{-\sigma}$, we get
\begin{align*}
	\mathcal{B} &= \langle k, kt \rangle^{\sigma} e^{-\frac{1}{4} \theta_1 \vert k \vert (t-s)} \leq 2^{\sigma}  \langle k, ks \rangle^{\sigma}  \langle 0, kt -ks \rangle^{\sigma} e^{-\frac{1}{4} \theta_1 \vert k \vert (t-s)}\\
	& \leq 2^{\sigma} C  \langle k, ks \rangle^{\sigma} \langle 0, kt -ks \rangle^{\sigma} \left(\frac{\theta_1}{4}\right)^{-\sigma} \langle k(t-s) \rangle^{-\sigma}\leq C \langle k, ks \rangle^{\sigma},
\end{align*}
for some constant $C \geq 0.$ Hence, we have proved \eqref{eq:est_theta_1}.

Similarly, we can show that 
\begin{equation*}
	 A_{k,kt}   \int_0^t e^{-\theta_1 \vert k \vert (t-s)} \vert \widehat{h(U)}_k(s) \vert \ ds  \leq C\int_0^t e^{-\frac{\theta_1}{4} (t-s)} A_{k,ks} \vert \widehat{h(U)}_k(s) \vert \ ds.
\end{equation*}
Therefore, we have 
\begin{align*}
	A_{k,kt} \vert \widehat{\rho}_{k} (t)  \vert  &\leq A_{k,kt}  \vert \widehat{f^0}_{k,kt} \vert + C   \int_0^t  e^{-\frac{\theta_1}{4}  (t-s)} A_{k,ks} \vert \widehat{f^0}_{k,ks} \vert \ ds + C \int_0^t e^{-\frac{\theta_1}{4} (t-s)} A_{k,ks} \vert \widehat{h(U)}_k(s) \vert \ ds.
\end{align*}
Next, multiplying by $\frac{1}{\vert k \vert^\alpha}$ and taking the supremum over $k \neq 0$, we obtain
\begin{align*}
	\sup_{k \in \mathbb{Z}^3 \setminus \{0\}} A_{k,kt} \vert \widehat{\rho}_{k} (t)  \vert  \frac{1}{\vert k \vert^\alpha}  &\leq \sup_{k \in \mathbb{Z}^3 \setminus \{0\}} A_{k,kt}  \vert \widehat{f^0}_{k,kt} \vert \frac{1}{\vert k \vert^\alpha} + C   \int_0^t  e^{-\frac{\theta_1}{4}  (t-s)} \sup_{k \in \mathbb{Z}^3 \setminus \{0\}} A_{k,ks} \vert \widehat{f^0}_{k,ks} \vert \frac{1}{\vert k \vert^\alpha}  \ ds \\
	& + C \int_0^t e^{-\frac{\theta_1}{4} (t-s)} \sup_{k \in \mathbb{Z}^3 \setminus \{0\}} A_{k,ks} \vert \widehat{h(U)}_k(s) \vert \frac{1}{\vert k \vert^\alpha} \ ds.
\end{align*}
By definition of the generator function (\ref{GenF})\footnote{By abuse of notation, we define the generator function $F[...]$ of the function $f^0(x,v)$ as
\begin{equation*}
	F[f^0](t,z) = \sup_{k \in \mathbb{Z}^3 \setminus \{0\}} A_{k,kt}  \vert \widehat{f^0}_{k,kt} \vert \frac{1}{\vert k \vert^\alpha}.
\end{equation*}} , we get
\begin{align*}
 	 F[\rho](t,z)  & \leq   F[f^0](t,z) + C   \int_0^t  e^{-\frac{\theta_1}{4}  (t-s)}  F[f^0](s,z)  \ ds + C   \int_0^t  e^{-\frac{\theta_1}{4}  (t-s)}  F[h(U)](s,z)  \ ds \\ 
 	 & \leq  C  \sqrt{\varepsilon} + C   \int_0^t  e^{-\frac{\theta_1}{4}  (t-s)}  \sqrt{\varepsilon}  \ ds + C   \int_0^t  e^{-\frac{\theta_1}{4}  (t-s)}  F[h(U)](s,z)  \ ds.
\end{align*}
The last inequality follows from Lemma \ref{lemma_F<G} and the assumption (\ref{IC}), Indeed, since  $z \leq \lambda_1 $ we have,
\begin{equation*}
	F[f^0](t,z) \leq C G[f^0]^{\frac{1}{2}}(z) \leq C G[f^0]^{\frac{1}{2}}(\lambda_1)  \leq C \sqrt{\varepsilon}.
\end{equation*}
By Lemma \ref{property_F_2}, i.e. $ F[h(U)](t,z)  \leq \widetilde{h}\Big(C F[U](t,z)\Big)$ and taking the $\sup_{s\leq t}  \widetilde{h}(CF[U](s,z))$ out of the integral we get,
\begin{align*}
 	 F[\rho](t,z)  	& \leq C \sqrt{\varepsilon} + C   \int_0^t  e^{-\frac{\theta_1}{4}  (t-s)}   \widetilde{h}\Big(C F[U] (s,z)\Big)  \ ds  \leq  C \sqrt{\varepsilon} + C \sup_{s\leq t}   \widetilde{h}\Big(C F[U](s,z)\Big),
\end{align*}
where we used 
\begin{equation*}
	\int_0^t  e^{-\frac{\theta_1}{4}  (t-s)}  \ ds \leq C,
\end{equation*}
for some constant $C$.
\end{proof}

From (\ref{estimate_F_rho}), we see that if we want to obtain a uniform bound on $F[\rho](t,z)$, we need to estimate $F[U](t,z)$. It is the goal of the next section.

\subsection{Estimate on the norm \texorpdfstring{$F[U](t,z)$}{Lg}}
\label{Estimate on F_U}

To estimate $F[U](t,z),$ we start deriving a closed equation on $\widehat{U}_k (t)$ as done in Section \ref{Equation on the density}.

\begin{lem}
We have the following equation for $\widehat{U}_k (t)$:
\begin{equation}
\label{U_hat_1}
	\widehat{U}_k (t) + \frac{\vert k \vert^2}{\beta + \vert k \vert^2} \int_0^t  \widehat{U}_k (s) (t-s) \widehat{ \mu } (k (t - s)) \ ds =  \widehat{T}_k(t),
\end{equation}
where the source term $\widehat{T}_k(t)$ is given by
\begin{equation}
\label{def_T_k}
 	\widehat{T}_k(t) := \frac{1}{\beta + \vert k \vert^2} \widehat{f^0}_{k,kt} - \frac{1}{\beta + \vert k \vert^2} \widehat{h(U)}_k (t). 
 \end{equation} 

\end{lem}

\begin{proof}
The proof starts as Lemma \ref{lemma_rhohat} until we get the equation (\ref{rho_hat_aux2}), i.e.
\begin{equation*}
	\widehat{\rho}_{k} (t)  =  \widehat{f^0}_{k,kt} - \int_0^t ik \cdot  \widehat{E}_k (s)  (t-s) \widehat{ \mu } (k(t - s)).
\end{equation*}
Then, plugging equations (\ref{E_hat}) and (\ref{rho_hat}) in (\ref{rho_hat_aux2}), we obtain a closed equation on $\widehat{U}_k (t)$,
\begin{equation*}
	(\vert k \vert^2 + \beta) \widehat{U}_k (t) + \widehat{h(U)}_k(t)  =  \widehat{f^0}_{k,kt} - \int_0^t \vert k \vert^2 \widehat{U}_k (s)  (t-s) \widehat{ \mu } (k(t - s)),
\end{equation*}
which can be rewritten as follows:
\begin{equation*}
	\widehat{U}_k (t) + \frac{\vert k \vert^2}{\beta + \vert k \vert^2} \int_0^t  \widehat{U}_k (s) (t-s) \widehat{ \mu } (k (t - s)) \ ds = \frac{1}{\beta + \vert k \vert^2} \widehat{f^0}_{k,kt} - \frac{1}{\beta + \vert k \vert^2} \widehat{h(U)}_k (t).
\end{equation*}
\end{proof}
Performing the same analysis on (\ref{U_hat_1}) as we did above for the closed equation on $\widehat{\rho}_{k} (t)$, we obtain the following lemma.
\begin{lem} 
Let $\widehat{U}_k (t)$ be a solution of (\ref{U_hat_1}) and let $\widehat{K}_k (t)$ be the inverse Laplace transform of the resolvent kernel (\ref{K_hat}). Then we can express $\widehat{U}_k (t)$ as 
\begin{equation}
\label{U_hat_2}
	\widehat{U}_k (t) = \widehat{T}_k(t) + \int_0^t \widehat{K}_k (t-s) \widehat{T}_k(s) \ ds,
\end{equation}
where $\widehat{T}_k(t)$ is given by (\ref{def_T_k})  and $ \widehat{K}_k (t)$ satisfies (\ref{inverse_resolvent}).
	
\end{lem}

\begin{proof}
We take the Laplace transform of (\ref{U_hat_1}), and we get
\begin{equation*}
	\mathcal{L}[\widehat{U}_k](\lambda) + \frac{\vert k \vert^2}{\beta + \vert k \vert^2}  \mathcal{L}[t\widehat{\mu}(kt)](\lambda) \mathcal{L}[\widehat{U}_k](\lambda) = \mathcal{L}[\widehat{T}_k](\lambda).
\end{equation*}
Then we use the definition of the resolvent kernel (\ref{resolvent}) to have
\begin{equation*}
	\mathcal{L}[\widehat{U}_k](\lambda)  = \mathcal{L}[\widehat{T}_k](\lambda) +  \widetilde{K}_k (\lambda) \mathcal{L}[\widehat{T}_k](\lambda).
\end{equation*}
Finally, we apply the inverse Laplace transform to the last equality to end up with 
 \begin{equation*}
 	\widehat{U}_k (t) = \widehat{T}_k(t) + \int_0^t \widehat{K}_k (t-s) \widehat{T}_k(s) \ ds.
 \end{equation*}
\end{proof}
The next lemma gives a first bound on $F[U](t,z)$.
\begin{lem}
\label{lemma estimate F_U}
Assume  $\widehat{U}_k (t)$ satisfies (\ref{U_hat_2}). Then, for every time $t >0$ and $z \in [0, \frac{\theta_1}{2}] \cap [0,  \lambda_1]$ with $\theta_1$ defined in Lemma \ref{lemmaresolvent} and $\lambda_1$ defined in Theorem \ref{Thm1}, we have
\begin{equation}
\label{estimate F_U}
	F[U](t,z) \leq C \sqrt{\varepsilon}  + C \sup_{s \leq t} \widetilde{h}\Big(C F[U](s,z)\Big).
\end{equation}
\end{lem}

\begin{proof}
We proceed as in the proof of Proposition \ref{Prop_estimate_F_rho}.
First, using (\ref{U_hat_2}) and (\ref{inverse_resolvent}) we have
\begin{equation*}
	A_{k,kt} \vert \widehat{U}_k (t) \vert \leq  A_{k,kt} \vert  \widehat{T}_k(t) \vert + CA_{k,kt}  \int_0^te^{-\theta_1 \vert k \vert (t-s)}  \vert  \widehat{T}_k(s) \vert \ ds.
\end{equation*}
Next, we can show exactly as in the  proof of Proposition \ref{Prop_estimate_F_rho} that 
\begin{equation*}
	A_{k,kt}  \int_0^t   e^{-\theta_1 \vert k \vert (t-s)}  \vert  \widehat{T}_k(s) \vert \ ds \leq C\int_0^t  e^{-\frac{\theta_1}{4}  (t-s)} A_{k,ks}  \vert  \widehat{T}_k(s) \vert \ ds.
\end{equation*}
Therefore, we get
\begin{equation*}
	A_{k,kt} \vert \widehat{U}_k (t) \vert \leq  A_{k,kt} \vert  \widehat{T}_k(t) \vert + C\int_0^t  e^{-\frac{\theta_1}{4}  (t-s)} A_{k,ks}  \vert  \widehat{T}_k(s) \vert \ ds.
\end{equation*}
Then we multiply both side by $\frac{1}{\vert k \vert^\alpha}$ and we take the supremum over $k \neq 0$.
\begin{equation*}
	\sup_{k \in \mathbb{Z}^3 \setminus \{0\}}  A_{k,kt} \vert \widehat{U}_k (t) \vert \frac{1}{\vert k \vert^\alpha} \leq \sup_{k \in \mathbb{Z}^3 \setminus \{0\}}  A_{k,kt} \vert  \widehat{T}_k(t) \vert \frac{1}{\vert k \vert^\alpha} + C\int_0^t e^{-\frac{\theta_1}{4}  (t-s)} \sup_{k \in \mathbb{Z}^3 \setminus \{0\}}  A_{k,ks}  \vert  \widehat{T}_k(s) \vert  \frac{1}{\vert k \vert^\alpha} \ ds.
\end{equation*}
Therefore, 
\begin{equation}
\label{F_U_2}
	 F[U](t,z)   \leq  F[T](t,z) + C   \int_0^t  e^{-\frac{\theta_1}{4}  (t-s)}  F[T](s,z) \ ds.
\end{equation}
Recalling (\ref{def_T_k}), we find 
\begin{align}
\label{Bound_F_T}
	F[T](t,z)  \leq F[f^0](t,z) + F[h(U)](t,z) \leq C \sqrt{ \varepsilon} + \widetilde{h} \Big( C F[U](t,z) \Big),
\end{align}
where we have used Lemma \ref{property_F_2} to have $F[h(U)](t,z) \leq  \widetilde{h} \Big( C F[U](t,z) \Big)$ and Lemma \ref{lemma_F<G} to have, for $z \leq \lambda_1,$
\begin{equation*}
	F[f^0](t,z) \leq C  G[f^0]^{\frac{1}{2}}(z) \leq C G[f^0]^{\frac{1}{2}}(\lambda_1)  \leq C \sqrt{\varepsilon}.
\end{equation*}
Finally, plugging the bound (\ref{Bound_F_T}) in inequality (\ref{F_U_2}):
\begin{align*}
 F[U](t,z)   & \leq C  \sqrt{ \varepsilon} + \widetilde{h} \Big( C F[U](t,z) \Big) + C   \int_0^t  e^{-\frac{\theta_1}{4}  (t-s)} \left(\sqrt{ \varepsilon} + \widetilde{h} \Big( C F[U](s,z) \Big) \right) \ ds \\ 
	  %& \leq C \sqrt{ \varepsilon} + \widetilde{h} \Big( C F[U](t,z) \Big)  + C \sup_{s \leq t} \widetilde{h} \Big( C F[U](t,z) \Big) \\
	  & \leq C \sqrt{\varepsilon}  + C \sup_{s \leq t} \widetilde{h}\Big(C F[U](s,z)\Big).
\end{align*}
\end{proof}
In the following lemma, we first show the smallness of $F[U](t,z)$ at time $t=0$, and then, by a continuity argument, we prove that $F[U](t,z)$ remains small for all times.
\begin{lem}
\label{lemma_F_U}
Let $\widehat{U}_k (t)$ be a solution of (\ref{U_hat_2}). Then there exists a constant $C>0$ such that
\begin{equation}
\label{F_U}
	F[U](t,z) \leq  C \sqrt{\varepsilon} ,
\end{equation}
for every $z \in [0, \frac{\theta_1}{2}] \cap [0,  \lambda_1]$  and $t>0.$ 
\end{lem}

\begin{proof}
By Lemma \ref{lemma estimate F_U} and Lemma \ref{Assumptions on the initial datum 2}, we have the following hypothesis on $F[U](t,z):$
\begin{eqnarray*}
    \left \{
    \begin{array}{l}
   F[U](t,z) \leq C \sqrt{\varepsilon}  + C\sup_{s \leq t} \widetilde{h} \Big( C F[U](s,z) \Big), \\
    F[U](0,z) \leq   C_1 \sqrt{\varepsilon}.  
    \end{array} 
    \right.
\end{eqnarray*}
Let us define $Y(t) := F[U](t,z)$ and $\overline{Y}(t) := \sup_{s \leq t} C F[U](s,z)$. Let us observe that both $\overline{Y}$ and $\widetilde{h}$ are increasing functions. Taking the supremum on both sides, we get
\begin{eqnarray}
\label{system Y(t)}
    \left \{
    \begin{array}{l}
   \overline{Y}(t) \leq  C \sqrt{\varepsilon} + C \widetilde{h} \left(\overline{Y}(t)\right), \\
    \overline{Y}(0) \leq C_1 \sqrt{\varepsilon}.
    \end{array}
    \right.
\end{eqnarray}
Since $\overline{Y}(0) \leq  C_1 \sqrt{\varepsilon} \ll 1$, given any $\delta> C_1 \sqrt{\varepsilon}$ small, by continuity there exists a time $t> 0$ such that $ \overline{Y}(t) < \delta \ll 1$. Then by the assumption: $\widetilde{h}(x) = \mathcal{O}(x^2)$ as $x$ tends to zero, we have 
\begin{equation*}
	\widetilde{h}\Big(\overline{Y}(t) \Big) \leq   C \left( \overline{Y}(t) \right)^2 \leq C_0 \delta \overline{Y}(t).
\end{equation*}
Thus, we have
$$
\overline{Y}(t) \leq  C \sqrt{\varepsilon} + C_0 \delta \overline{Y}(t).
$$ 
Choosing $\delta =  \frac{1}{2C_0} $ we get $\overline{Y}(t) \leq C \sqrt{\varepsilon} + \frac{1}{2} \overline{Y}(t)$, which implies
 \begin{equation*}
	\overline{Y}(t) \leq 2 C \sqrt{\varepsilon},
\end{equation*}
for any $t$ such that $\overline{Y}(t)<\delta.$
Now, let us assume that $T^* < \infty$ is the first time when $\overline{Y}(T^*) = \delta.$ Then, if $\sqrt{\varepsilon} < \frac{\delta}{ 2 C}$, we have 
\begin{equation*}
	 \delta = \overline{Y}(T^*) \leq 2C \sqrt{\varepsilon} < \delta,
\end{equation*}
that contradicts $T^* < \infty$, and thus implies $T^* = \infty$.
Therefore, for every time $t>0$, we have
$$
\overline{Y}(t) \leq   2 C \sqrt{\varepsilon}.
$$ 
Finally, recalling the definition of $\overline{Y}$ we proved that:
\begin{equation*}
	F[U](t,z) \leq  C \sqrt{\varepsilon}, 
\end{equation*}
for every $z \in [0, \frac{\theta_1}{2}] \cap [0,  \lambda_1]$ and $t>0$.
\end{proof}

Thanks to Lemma \ref{lemma_F_U}, we can now bound the generator function $F[\rho](t,z)$.
\begin{prop}
\label{bdd F_rho}
Assume that $\widehat{\rho}_k (t)$ satisfies equation (\ref{rhohat3}), and let $\widehat{U}_k (t)$ be a solution of (\ref{U_hat_2}). Then, there exists a constant $C>0$ such that
\begin{equation}
	F[\rho](t,z) \leq C \sqrt{\varepsilon},
\end{equation}
for every time $t \geq 0$ and $z \in [0, \frac{\theta_1}{2}] \cap [0,  \lambda_1]$ with $\theta_1$ defined in Lemma \ref{lemmaresolvent} and $\lambda_1$ defined in Theorem \ref{Thm1}.
\end{prop}

\begin{proof}
We need to combine inequalities (\ref{estimate_F_rho}) and (\ref{F_U}), 
\begin{equation*}
	F[\rho](t,z) \leq  C \sqrt{\varepsilon} + C \sup_{s\leq t} \widetilde{h}\Big(C F[U](s,z)\Big) \leq  C \sqrt{\varepsilon} + C  \widetilde{h}\Big(C \sqrt{\varepsilon}\Big)  \leq C \sqrt{\varepsilon} +  C \varepsilon \leq C \sqrt{\varepsilon},
\end{equation*}
where we have used that $\widetilde{h}$ is increasing on $\br_+$ and $\widetilde{h}(x) = \mathcal{O}(x^2)$ as $x$ tends to zero.

\end{proof}

\subsection{Asymptotic behaviour}
\label{Asymptotic behaviour}

In this section, we show that the density  and the electric field go to zero as $t$ goes to $ +\infty$, and we prove that the solution $f$ converges weakly to a limit $f_{\infty}$. Namely, we close the proof of Landau damping for the linear system (\ref{lpvpme}).

First we give a quick recap about the values of  $\theta_0$, $\theta_1$, $\lambda_1$ and $z$. The constant $\theta_0$ is in the exponential decay of the equilibrium in Fourier space, i.e. $\vert \widehat{\mu}(\eta) \vert \leq C e^{-\theta_0 \vert \eta \vert}$.  The constant $ \theta_1$ comes from Lemma \ref{lemmaresolvent} with $\vert \widehat{K}_k (t) \vert \leq C e^{- \theta_1 \vert kt \vert}$ and $\theta_1 < \frac{1}{2} \theta_0$.   The radius of analyticity of the initial condition $f^0$ is denoted by $\lambda_1$, and we have $G[f^0](\lambda_1) \leq \varepsilon$.  The radius of analyticity of our solution is denoted by $z$. Up to this point, we have used that $z \in \left[0, \frac{\theta_1}{2}\right] \cap [0, \lambda_1]$.

Since $f^0$ is analytic or Gevrey with index\footnote[1]{The restriction on  $\gamma \in (\frac{1}{3} ,1]$ concerns only the nonlinear part. For the linear part we can take $\gamma \in (0,1].$}  $\gamma \in (0, 1]$, we know that there exist constants $C >0$ and $b >0$ such that 
\begin{equation*}
	\vert \widehat{f^0}_{k,\eta}  \vert \leq Ce^{-b \vert \eta \vert^{\gamma}}.
\end{equation*}

Proposition \ref{bdd F_rho} shows the exponential decay of the density $\rho$. Indeed, by definition of the generator function (\ref{GenF}) and the uniform bound in time, we have
\begin{equation*}
	F[\rho](t,z) =  \sup_{k \in \mathbb{Z}^3 \setminus \{0\}} e^{z \langle k, kt \rangle^{\gamma}} \vert \widehat{\rho}_k (t) \vert \langle k, kt \rangle^{ \sigma}  \frac{1}{\vert k \vert^{ \alpha }}  \leq C.
\end{equation*}
Which implies that $\vert \widehat{\rho}_k (t) \vert \leq e^{-r \langle k, kt \rangle^{\gamma}}$ for some $r \geq z$. Then, using the formula for Fourier series (\ref{reconstruction formula 1}) and the inequality $\langle k, kt \rangle^{\gamma} \geq \langle t \rangle^{\gamma} + \langle k\rangle^{\gamma}$ that holds for large $t$ and for $k \neq 0$, we get
\begin{equation*}
	\vert \rho (t,x) \vert = \left|  \frac{1}{(2\pi)^3} \sum_{k \in \mathbb{Z}^3  \setminus \{0\}} \widehat{\rho}_k (t) e^{ik \cdot x}  \right| \leq C\sum_{k \in \mathbb{Z}^3  \setminus \{0\}} e^{-r \langle k, kt \rangle^{\gamma}} \lesssim C e^{-r \langle t \rangle^{\gamma}} \sum_{k \in \mathbb{Z}^3  \setminus \{0\}} e^{-r \langle k\rangle^{\gamma}} \leq Ce^{-r \langle t \rangle^{\gamma}} 
\xrightarrow[t \rightarrow \infty]{}  0.
\end{equation*}

By Lemma \ref{lemma_F_U} we have the same result for the potential $U$. That is, $U$ decays exponentially fast to zero when $t$ goes to infinity. Hence, the electric field goes to zero as well at exponential rate. Let us write $r_0$ the rate such that $\vert \widehat{E}_k (t) \vert \leq C e^{-r_0 \langle k, kt \rangle^\gamma}$.

We can show the scattering estimate (\ref{scaterring}) in the linear case as well. Nevertheless, we refer to Section \ref{section bound_rho} for a proof of the scattering to free transport of the linear evolution. The argument is similar to  nonlinear computations, except that we do not have the second integral in (\ref{g integral}).

From the scattering estimate (\ref{scaterring}), we can deduce that $f(t,x,v)$ converges weakly to its spatial average.  Here, we prove this convergence using directly the equation on $f$ and the estimates obtain above. For this we go back to equation (\ref{equ_g_hat}), i.e. 
\begin{equation*}
	\widehat{g}_{k,\eta} (t) = \widehat{g}_{k,\eta} (0) - \int_0^t \widehat{E}_k (s) \cdot \widehat{\nabla_v \mu } (\eta - ks) \ ds.
\end{equation*}
Using $g(t,x,v)= f(t,x + vt,v)$, we can rewrite this equality in term of $f$ as
\begin{equation*}
	\widehat{f}_{k,\eta - kt} (t) = \widehat{f^0}_{k,\eta} - \int_0^t \widehat{E}_k (s) \cdot \widehat{\nabla_v \mu } (\eta - ks) \ ds.
\end{equation*}
We remark that for $k=0$ we have $\widehat{f}_{0,\eta} (t) =  \widehat{f^0}_{0,\eta}$ because $\widehat{E}_0 (s) = 0$. Therefore, $$
\langle f \rangle := \int_{\mathbb{T}^3} f(t,x,v) \ dx=\langle f^0 \rangle
$$
for all times. Indeed, using the reconstruction formula (\ref{reconstruction formula 2}) for $v$ we get
\begin{align*}
	\langle f \rangle &:= \int_{\mathbb{T}^3} f(t,x,v) \ dx =  \widehat{f}_0 (t,v)= \frac{1}{(2\pi)^3} \int_{\mathbb{R}^3} \widehat{f}_{0, \eta}(t) e^{i\eta \cdot v} \ d\eta \\
	&=\frac{1}{(2\pi)^3} \int_{\mathbb{R}^3} \widehat{f}_{0, \eta}^0 e^{i\eta \cdot v} \ d\eta =  \widehat{f^0}_{0} (v) = \langle f^0 \rangle,
\end{align*}
where $ \widehat{f}_0 (t,v)$ and $\widehat{f^0}_{0} (v)$ are Fourier transforms but only in the $x$ variable. Therefore, we can compute for  $k \neq 0$,
\begin{align*}
	\vert \widehat{f}_{k,\eta - kt} (t) \vert & \leq  \vert \widehat{f^0}_{k,\eta}  \vert + \int_0^t \vert \widehat{E}_k (s) \vert  \vert \widehat{\nabla_v \mu } (\eta - ks) \vert \ ds \\
	& \leq C e^{-b \vert \eta \vert^{\gamma}} + C  \int_0^t  e^{-r_0 \vert ks \vert^{\gamma}} e^{-\theta_0 \vert \eta - ks \vert} \ ds \\
	& \leq C e^{-b \vert \eta \vert^{\gamma}} + C  \int_0^{\infty}  e^{-r_0 \vert ks \vert^{\gamma}} e^{-\theta_0 \vert \eta - ks \vert} \ ds,
\end{align*}
where we have used the analytic or Gevrey regularity of $f^0$,   the decay of the electric field, and of the homogeneous equilibrium, next, we split the integral into two parts, $\left\lbrace \vert ks  \vert> \frac{\vert \eta \vert}{2}  \right\rbrace$ and $\left\lbrace \vert ks  \vert < \frac{\vert \eta \vert}{2}  \right\rbrace$.
\begin{align*}
	\int_0^{\infty}  e^{-r_0 \vert ks \vert^{\gamma}} e^{-\theta_0 \vert \eta - ks \vert} \ ds & =   \int_0^{\infty}  e^{-r_0 \vert ks \vert^{\gamma}} e^{-\theta_0 \vert \eta - ks \vert}  \mathds{1}_{ \left\lbrace \vert ks  \vert> \frac{\vert \eta \vert}{2}  \right\rbrace } \ ds \\
	& \qquad + \int_0^{\infty}  e^{-r_0 \vert ks \vert^{\gamma}} e^{-\theta_0 \vert \eta - ks \vert}  \mathds{1}_{ \left\lbrace \vert ks  \vert < \frac{\vert \eta \vert}{2}  \right\rbrace } \ ds \\
	& \leq \int_0^{\infty}  e^{-\frac{r_0}{2} \vert \eta \vert^{\gamma}} e^{-\theta_0 \vert \eta - ks \vert}  \ ds  + \int_0^{\infty}  e^{-r_0 \vert ks \vert^{\gamma}} e^{-\frac{\theta_0}{2} \vert \eta \vert} \ ds \\
	&  \leq C e^{-\frac{r_0}{2} \vert \eta \vert^{\gamma}} + C  e^{-\frac{\theta_0}{2} \vert \eta \vert},
\end{align*}
where we have applied the inverse triangle inequality in the first inequality for the second integral, i.e. $ \Big\vert \eta - ks \Big\vert \geq \Big\vert \vert \eta \vert - \vert ks \vert \Big\vert \geq \Big\vert \vert \eta \vert - \vert \frac{\eta}{2} \vert \Big\vert \geq \Big\vert \frac{\eta}{2} \Big\vert.$ Therefore, replacing the integral by this inequality we get
\begin{equation*}
	\vert \widehat{f}_{k,\eta - kt} (t) \vert \leq  C e^{-b \vert \eta \vert^{\gamma}} + C e^{-\frac{r_0}{2} \vert \eta \vert^{\gamma}} + C  e^{-\frac{\theta_0}{2} \vert \eta \vert} \leq C_{b, r_0, \theta_0} e^{-b \vert \eta \vert^{\gamma}},
\end{equation*}
where $C_{b,  r_0,  \theta_0}$ is a constant which depends on $b,  r_0, $ and $\theta_0$.
Hence we have for any fixed $(k, \eta) \in (\mathbb{Z}^3  \setminus \{0\}) \times \mathbb{R}^3$,
\begin{equation*}
	\vert \widehat{f}_{k,\eta} (t) \vert  \leq C_{b, r_0, \theta_0} e^{-b \vert \eta + kt \vert^{\gamma}}.
\end{equation*}
Let us define 
\begin{equation*}
	f_\infty(x,v) := f^0(x,v) - \int_0^\infty E(s,x+vs) \cdot \nabla_v \mu (v) \ ds.
\end{equation*}
Then,
\begin{equation*}
	\langle f_\infty (v) \rangle_x := \int_{\bt^d} f_\infty (x,v) \ dx = \int_{\bt^d} f^0 (x,v) \ dx,
\end{equation*}
because $\int_{\bt^d} E(x) \ dx = 0$.
Therefore, the Fourier transform of $f_\infty(x,v)$ is
\begin{equation*}
	\reallywidehat{\langle f_\infty (v) \rangle_x} (k,\eta) = \widehat{f^0}_{0,\eta} \delta_{k=0}.
\end{equation*}
Hence, for any fixed $(k, \eta) \in \mathbb{Z}^3 \times \mathbb{R}^3$, we get
\begin{align*}
	\vert \widehat{f}_{k,\eta} (t) -  \reallywidehat{\langle f_\infty (v) \rangle_x} (k,\eta) \vert  &= \vert \widehat{f}_{k,\eta} (t) -   \widehat{f^0}_{0,\eta} \delta_{k=0} \vert =  \left \{  \begin{array}{lr}
     \vert \widehat{f}_{0,\eta} (t) -  \widehat{f^0}_{0,\eta} \vert = 0 & \qquad \mbox{if} \qquad k=0 \\
    \vert \widehat{f}_{k,\eta} (t) \vert & \qquad \mbox{if} \qquad k \neq 0
    \end{array}  \right. \\
    \ \\
    & \leq C_{b, r_0, \theta_0} e^{-b \vert \eta + kt \vert^{\gamma}} \xrightarrow[t \rightarrow \infty]{} 0.
\end{align*}

\begin{prop}
\label{prop weak convergence}
Let $f$ be the solution above. Then $f$ converges weakly in $L^2 (\bt^d \times \br^d)$  to its spatial average as $t \to \infty$:
\begin{equation*}
	f(t,x,v) \xrightharpoonup[t \rightarrow \infty]{} \langle f_\infty (v) \rangle_x= \int_{\mathbb{T}^3} f^0(x,v) \ dx.
\end{equation*}
\end{prop}

\begin{proof}
Let $\varphi \in C_c^0 (\bt^d \times \br^3)$ be a given test function.  By Parseval's identity and Plancherel theorem, we have
\begin{align*}
	\langle \varphi (x,v),  f(t,x,v) \rangle_{L^2}  &= \langle  \widehat{\varphi }_{k,\eta}, \widehat{f}_{k,\eta}(t) \rangle_{L^2} = \int_{\br^3} \widehat{ \varphi}_{0,\eta}  \widehat{f^0}_{0,\eta} \ d\eta + \sum_{k \neq 0} \int_{\br^3} \widehat{\varphi }_{k,\eta} \widehat{f}_{k,\eta}(t) \ d \eta.
\end{align*}
The first term corresponds to our limit and the second term goes to zero as time goes to infinity thanks the dominated convergence Theorem and the decay of the Fourier transform.
Hence, we finally get
\begin{equation*}
	\lim_{t \rightarrow \infty} \int_{\bt^3 \times \br^3} \varphi (x,v) f(t,x,v) \ dx  \ dv =  \int_{\bt^3 \times \br^3} \varphi (x,v) \langle f_\infty (v) \rangle_x \ dx  \ dv.
\end{equation*}
\end{proof}

\section{Nonlinear Landau damping}
\label{section Nonlinear Landau damping}

Let us consider the nonlinear Vlasov equation
\begin{equation}
\label{nlpvpme2}
    \partial_t f + v \cdot \nabla_x  f  + E[f] \cdot \nabla_v (f + \mu) = 0,
\end{equation}
where the force field is given by
\begin{equation}
\label{U2}
    E = - \nabla U \quad \mbox{and} \quad - \Delta U + \beta U + h(U) = \rho.
\end{equation}
Following the characteristics of the free transport, and using again the notation $g(t,x,v) := f(t,x+vt,v)$, we get
\begin{equation*}
    \partial_t g = \partial_t f + v \cdot \nabla_x f, \quad \nabla_x g = \nabla_x f,
\end{equation*}
hence (\ref{nlpvpme2}) becomes
\begin{equation}
\label{nlequation_g}
    \partial_t g + E(t,x+vt) \cdot \nabla_v \mu = - E(t,x+vt) \cdot (\nabla_v g - t \nabla_x g).
\end{equation} 
The strategy of the proof follows \cite[ Section 4]{GNR}.  We first show a similar inequality to (4.1) in \cite{GNR} but with $\Delta U$ instead of $\rho$, i.e.
\begin{equation*}
 	\partial_t G[g(t)](z) \leq C F[\Delta U](t,z) G[g(t)]^{\frac{1}{2}} (z) + C (1+t)  F[\Delta U](t,z)  \partial_z G[g(t)] (z).
 \end{equation*} 
This inequality corresponds to \eqref{difeqG_U} below.  
Then, we show that if $F[\Delta U] (t, \lambda(t)) \leq 4C_1 \sqrt{\varepsilon}  \langle t \rangle^{- \sigma + 1}$,  with $\lambda (t) > 0$ defined later, we can bound the generator function of $g(t)$ as follows:
\begin{equation*}
	G[g(t)] (\lambda(t)) \leq 4 C_2 \varepsilon.
\end{equation*}
Secondly, we prove that if  $\sup_{t \leq T} G[g(t)] (\lambda(t)) \leq 4C_2 \varepsilon $, it implies 
\begin{equation*}
	F[\Delta U] (t, \lambda(t)) \leq 2C_1 \sqrt{\varepsilon}  \langle t \rangle^{- \sigma + 1},
\end{equation*}
for every time $t \in [0,T]$. Then by a continuity argument, we can get a uniform bounds for every time $t>0$
\begin{equation*}
	G[g(t)] (\lambda(t)) \leq 4C_2 \varepsilon,
\end{equation*}
Hence, using Lemma \ref{lemma_F<G} we can bound the generator function of $\rho$ with the one of $g$. That is,
 \begin{equation*}
 	F[\rho] (t,\lambda(t)) \leq  C G[g(t)]^{\frac{1}{2}} (\lambda(t)) \leq C \sqrt{\varepsilon}.
 \end{equation*}

\subsection{Differential inequality on  \texorpdfstring{$G[g(t)](z)$}{Lg}}
\label{Section bound_on_F[DeltaU]}
We have the following differential inequality on $G[g(t)](z)$.
\begin{prop}
\label{prop_difeqG_U}
There is a constant $C$ such that, for all $t$ and $z$, it holds 
\begin{align}
\label{difeqG_U}
\partial_t G[g(t)](z) \leq C F[\Delta U](t,z) G[g(t)]^{\frac{1}{2}} (z) + C (1+t)  F[\Delta U](t,z)  \partial_z G[g(t)] (z).
\end{align}
\end{prop}

\begin{proof}
This follows by the very same proof as the one of \cite[Proposition 4.1]{GNR}.
\end{proof}

\subsection{Bound on \texorpdfstring{$G[g(t)](z)$}{Lg}}
\label{section bound_g}

In this section, we show that the generator function of $g(t)$ is bounded.  Using inequality (\ref{difeqG_U}), we follow the analysis done in \cite[Section 4.3]{GNR}. 

The analyticity radius of our solution will decrease with time.  So we introduce a time dependent analyticity radius $z(t) = \lambda (t) z$ with $\lambda(t) = \lambda_0 + \lambda_0 (1 + t)^{-\delta}$ where the constant $\lambda_0$ and $\delta$ will be defined below.
We define the new generator functions $\widetilde{G}[g]$ and $\widetilde{F}[\Delta U]$ for a time dependent $z$,
\begin{equation*}
	\widetilde{G}[g(t)](z) = G[g(t)]( \lambda (t) z), \qquad \mbox{and} \qquad \widetilde{F}[\Delta U](t, z) = F[\Delta U](t, \lambda (t) z).
\end{equation*}
We note that,
\begin{equation*}
	\partial_t \widetilde{G}[g(t)](z) = \partial_t G[g(t)]( \lambda(t) z) + \partial_z G[g(t)]( \lambda(t) z)  \lambda'(t) z,
\end{equation*}
and
\begin{equation*}
	\partial_z \widetilde{G}[g(t)](z) = \partial_z G[g(t)]( \lambda(t)z)  \lambda(t).
\end{equation*}
So we can rewrite the differential inequality (\ref{difeqG_U}) with these generator functions: for $z \in [0,1]$, we have
\begin{align*}
	 \partial_t \widetilde{G}[g(t)]( z)  &  \leq C  F[\Delta U] (t, \lambda (t) z) G[g(t)]^{\frac{1}{2}} ( \lambda(t)z) + C (1+t)   F[\Delta U](t, \lambda (t) z)  \partial_z G[g(t)] (\lambda(t)z)\\
	 & \qquad \qquad + \partial_z G [g(t)]( \lambda(t) z) \ \lambda'(t) z\\
	& \leq C   \widetilde{F}[\Delta U](t, z)  \widetilde{G}[g(t)]^{\frac{1}{2}} (z) + \bigg( \frac{\lambda' (t) z + C (1+t) \widetilde{F}[\Delta U](t, z)}{\lambda (t)} \bigg)  \partial_z \widetilde{G}[g(t)] (z),
\end{align*}
that is,
\begin{align}
	\label{difeqG'_U}
\partial_t \widetilde{G}[g(t)] (z) & - \bigg( \frac{\lambda' (t) z + C (1+t)  \widetilde{F}[\Delta U](t, z) }{\lambda (t)} \bigg)  \partial_z \widetilde{G}[g(t)] (z) \leq C   \widetilde{F}[\Delta U](t, z)  \widetilde{G}[g(t)]^{\frac{1}{2}} (z).
\end{align}
We observe that (\ref{difeqG'_U}) is a transport equation on $\widetilde{G}[g]$.  If we can show that the characteristic curves are outgoing at $z=0$ and $z=1$, then we can bound $\widetilde{G}[g(t)] (z)$ for $z \in [0,1]$  by its initial value, i.e. by $\widetilde{G}[g(0)] (z) = \widetilde{G}[f^0] (z)$.

At $z=0$, we have
\begin{equation*}
	\bigg( \frac{\lambda' (t) z + C (1+t) \widetilde{F}[\Delta U](t,z)}{\lambda (t)} \bigg) \bigg\rvert_{z=0}  \geq 0,
\end{equation*}
since each term inside the brackets is nonnegative at $z=0$.

At $z=1$, we have to show that 
\begin{equation*}
	\bigg( \frac{\lambda' (t) z + C (1+t) \widetilde{F}[\Delta U](t,z)}{\lambda (t)} \bigg) \bigg\rvert_{z=1}  \leq 0.
\end{equation*}

\begin{lem}
\label{lem_F[DeltaU]}
Let $\lambda(t) = \lambda_0 + \lambda_0 (1 + t)^{-\delta}$ for some positive constant $\lambda_0$ such that $\lambda_0 \leq 1$ and $\lambda_0 \leq \frac{\lambda_1}{4}$ and for an arbitrary small positive constant $\delta \ll 1$ such that $3 \gamma > 1 + 2 \delta$, where $\gamma$ is the Gevrey index.
Assume that
\begin{equation}
\label{U_epsi}
	F[\Delta U] (t, \lambda(t)) \leq 4 C_1 \sqrt{\varepsilon}  \langle t \rangle^{- \sigma + 1}.
\end{equation}
Then
\begin{equation*}
	\frac{\lambda' (t)  + C (1+t)  F [\Delta U] (t, \lambda(t))}{\lambda (t)}   \leq 0,
\end{equation*}
provided $\sigma > 3 + \delta$.
\end{lem}
\begin{rem}
Note that, at time $t=0$, by Lemma \ref{Assumptions on the initial datum} we have
\begin{align*}
	F[\Delta U] (0, \lambda_1) \leq C_1 \sqrt{\varepsilon} \leq  4 C_1 \sqrt{\varepsilon}.
\end{align*}
Hence, (\ref{U_epsi}) is true for small times.
\end{rem}
\begin{proof}[Proof of Lemma \ref{lem_F[DeltaU]}]Since the denominator is always positive we only have to care about the numerator.
By (\ref{U_epsi}), we get
\begin{align*}
	\lambda' (t)  + C (1+t)  F[\Delta U](t, \lambda (t)) &  \leq \lambda' (t)  + C (1+t) 4 C_1 \sqrt{\varepsilon}  \langle t \rangle^{- \sigma + 1}  \\
	&  \leq - \delta \lambda_0 (1 + t)^{-\delta - 1} + C (1 + t) \sqrt{\varepsilon} (1 + t)^{- \sigma + 1}\\
		 & = - \delta \lambda_0 (1 + t)^{-\delta - 1} + C  \sqrt{\varepsilon} (1 + t)^{- \sigma + 2}  \leq 0,
\end{align*}
provided $-\delta -1 > - \sigma + 2$, which is equivalent to $\sigma > 3 + \delta$.
\end{proof}

Next, we consider the characteristic curve $(s, z(s))$ with $s \in [0, t]$  and $(t, z(t)) =(t, z)$ and such that
\begin{equation}
\label{caracteristic curve}
	z'(s) := - \frac{\lambda' (s) z(s) + C (1+s) \widetilde{F}[\Delta U](s,z(s)) }{\lambda (s)}.
\end{equation}
\begin{figure}[!ht] %[] ordre de preference du placement; ici h=here t=top b=bottom p=page of floats
\begin{center} %centrer la figure
\includegraphics[scale=2]{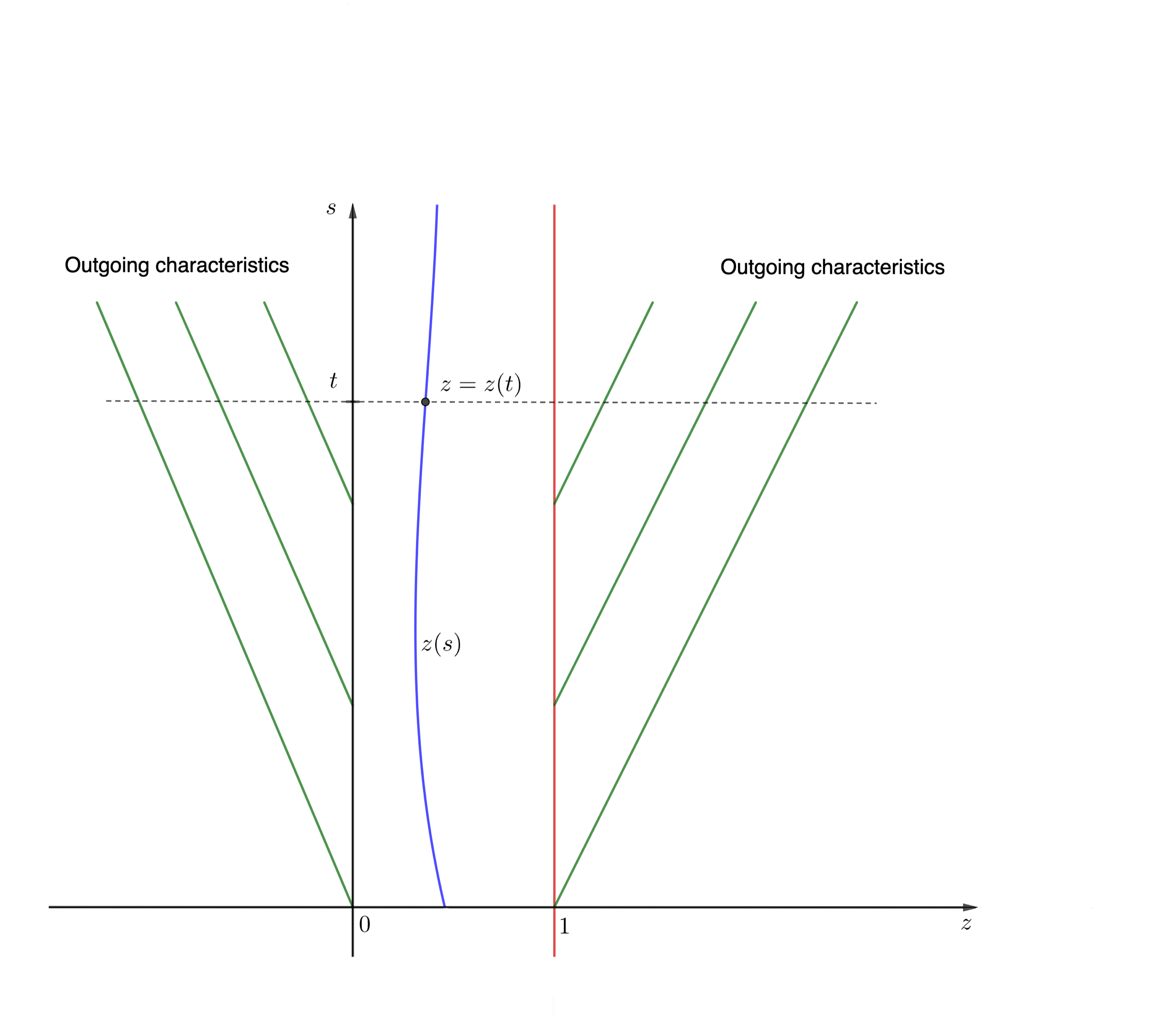} %le fichier doit se trouver au meme endroit que le fichier .tex
 %legende
\label{image2}
\caption{\label{image2_texte} Scheme of the characteristic curve $z(s)$. We showed that a $z=0$ and $z=1$ the curves are outgoing.}
\end{center}
\end{figure}
See Figure \ref{image2_texte} for a scheme of the curves. Then we can compute the total derivative of $\widetilde{G}[g(s)]^{\frac{1}{2}} (z(s))$,
\begin{align*}
	\frac{d}{ds} \widetilde{G}[g(s)]^{\frac{1}{2}} (z(s)) &= \frac{1}{2} \frac{\partial_s \widetilde{G}[g(s)] (z(s)) + \partial_z \widetilde{G}[g(s)] (z(s)) \ z'(s)}{\widetilde{G}[g(s)]^{\frac{1}{2}} (z(s))} \leq C\widetilde{F}[\Delta U](s,z(s)),
\end{align*}
where we used the differential inequality (\ref{difeqG'_U}) divided by $\widetilde{G}[g(s)]^{\frac{1}{2}} (z(s))$ and definition (\ref{caracteristic curve}).
Hence, we obtain 
\begin{equation*}
	\frac{d}{ds} \widetilde{G}[g(s)]^{\frac{1}{2}} (z(s)) \leq \sup_{0 \leq y \leq 1}C  \widetilde{F}[\Delta U](s,y).
\end{equation*}
We can integrate this equation,
\begin{equation*}
	\widetilde{G}[g(t)]^{\frac{1}{2}} (z(t))\leq \widetilde{G}[g(0)]^{\frac{1}{2}} (z(0)) +C \int_0^t \sup_{0 \leq y \leq 1}   \widetilde{F}[\Delta U](s,y) \ ds,
\end{equation*}
with $z(t) = z  \in [0,1]$.
Then by definition of $\widetilde{G}$ and $\widetilde{F}$, we get
\begin{equation*}
	G[g(t)]^{\frac{1}{2}} ( \lambda(t) z(t))\leq G[g(0)]^{\frac{1}{2}} (\lambda(0) z(0)) + C \int_0^t \sup_{0 \leq y \leq 1}    F[\Delta U](s,\lambda (s) y) \ ds.
\end{equation*}
Next, we use that $z(t)=z$, $\lambda(0) = 2 \lambda_0 \leq \lambda_1$ and the fact that $\widetilde{G}$ and $\widetilde{F}$ are two increasing functions with respect to $z$. Therefore, we have for every $z \in [0,1]$,
\begin{equation*}
	G[g(t)]^{\frac{1}{2}} ( \lambda(t) z)\leq G[g(0)]^{\frac{1}{2}} (\lambda_1) + C \int_0^t   F[\Delta U](s,\lambda (s))  \ ds.
\end{equation*}
Hence, for $z=1$ and using (\ref{U_epsi}), we get the following bound on $G$
\begin{multline*}
	G[g(t)]^{\frac{1}{2}} ( \lambda(t))  \leq G[g(0)]^{\frac{1}{2}} (\lambda_1) + C \int_0^t   F[\Delta U](s,\lambda (s)) \ ds 
	 \leq G[f^0]^{\frac{1}{2}} (\lambda_1) +  C  \cdot 4 C_1 \sqrt{\varepsilon}  \int_0^t   \langle s \rangle^{- \sigma + 1} \ ds \leq 4 C_2 \sqrt{\varepsilon},
\end{multline*}
where we used assumption (\ref{IC}) and that $ \int_0^t   \langle s \rangle^{- \sigma + 1} \ ds \leq C$ for $\sigma > 3 + \delta$, and $C_2$ is a constant such that $4 C_2 > C + C \cdot 4 C_1$. Hence, this proves that \begin{equation}
\label{G_epsi_U}
	G[g(t)] (\lambda(t)) \leq  4 C_2 \varepsilon.
\end{equation}

\subsection{Bound on \texorpdfstring{$F[\Delta U](t,z)$}{Lg}}
\label{Section bound_U}

Here,  we show that if (\ref{G_epsi_U}) is true for a small $\varepsilon$, then (\ref{U_epsi}) is satisfied with a factor term equals to 2 instead of 4.  Then we can conclude our analysis using a continuity argument and obtain the smallness of the generator function of $\Delta U$ and $g$ for every $t>0$.

\begin{prop}
\label{prop_bdd_F_U}
Let $\gamma \in (\frac{1}{3},1]$,  $\sigma > 3 + \delta$ and let $\lambda_0 < \frac{\lambda_1}{4}$.  Let $T$ be the largest time such that  
\begin{equation}
\label{assumption G<2eps}
	\sup_{0 \leq t \leq T} G[g(t)](\lambda(t)) \leq 4 C_2 \varepsilon, 
\end{equation}
for some sufficiently small $\varepsilon$. Then for $t \in [0,T]$ and for some constant $C_1$, we have
\begin{equation*}
	 F[\Delta U] (t, \lambda (t))  \leq 2 C_1 \sqrt{\varepsilon}  \langle t \rangle^{- \sigma + 1}.
\end{equation*}
\end{prop}

\begin{proof}
We take the Fourier transform of equation (\ref{nlequation_g}),  and we obtain
\begin{equation}
\label{FT_g} \cdot
	\partial_t \widehat{g}_{k ,\eta  }(t) = - \widehat{E}_k (t) \cdot \widehat{\nabla_v \mu} (\eta - kt) - i \sum_{\ell \in \mathbb{Z}^3} (\eta -kt)\cdot \widehat{E}_{\ell} (t) \widehat{g}_{k - \ell ,\eta  - \ell t } (t).
\end{equation}
Integrating (\ref{FT_g}) in time leads to
\begin{equation*}
	\widehat{g}_{k ,\eta  } (t) - \widehat{g}_{k ,\eta  } (0) = - \int_0^t \widehat{E}_k (s) \cdot \widehat{\nabla_v \mu} (\eta - ks) \ ds  -\int_0^t  i \sum_{\ell \in \mathbb{Z}^3}  (\eta -ks) \cdot \widehat{E}_{\ell} (s) \widehat{g}_{k - \ell ,\eta  -\ell s} (s) \ ds.
\end{equation*}
Then,  by considering the Fourier mode $\eta = kt$ and recalling $\widehat{g}_{k,kt} (t) = \widehat{\rho}_k (t)$ and (\ref{E_mu}),
 we obtain
\begin{align*}
\widehat{\rho}_k (t) - \widehat{f^0}_{k,kt}  = - \int_0^t \widehat{E}_k (s)  \cdot ik(t-s) \widehat{ \mu } (k(t - s)) \ ds  - \int_0^t i \sum_{\ell \in \mathbb{Z}^3} (k(t -s)) \cdot \widehat{E}_{\ell} (s) \widehat{g}_{k - \ell ,kt  - \ell s} (s) \ ds.
\end{align*}
Therefore, using (\ref{E_hat}) and (\ref{rho_hat}), we have 
\begin{multline}
\label{U_hat_nonlinear}
	 \widehat{U}_k(t) + \frac{1}{\beta + \vert k \vert^2} \int_0^t \vert k \vert^2  \widehat{U}_{k}(s) (t-s)   \widehat{ \mu } (k(t - s)) \ ds = \frac{1}{\beta + \vert k \vert^2}\widehat{f^0}_{k,kt} - \frac{1}{\beta + \vert k \vert^2}  \widehat{h(U)}_k (t)  \\
	 - \frac{1}{\beta + \vert k \vert^2}  \int_0^t \sum_{\ell \in \mathbb{Z}^3 \setminus \{0\}} k(t  - s) \cdot  \ell \widehat{U}_{\ell}(s)  \widehat{g}_{k - \ell ,kt  - \ell s} (s) \ ds.
\end{multline}
We define 
\begin{align}
\label{T_hat}
	\widehat{T}_k(t) &:= \frac{1}{\beta + \vert k \vert^2}\widehat{f^0}_{k,kt} - \frac{1}{\beta + \vert k \vert^2}  \widehat{h(U)}_k (t) -   \frac{1}{\beta + \vert k \vert^2} \int_0^t \sum_{\ell \in \mathbb{Z}^3 \setminus \{0\}}  k(t  - s) \cdot  \ell \widehat{U}_{\ell}(s)   \widehat{g}_{k - \ell ,kt  - \ell s} (s) \ ds.
\end{align}
Since we want to bound $F[\Delta U]$, we multiply (\ref{U_hat_nonlinear}) by $\vert k \vert^2$:
\begin{align*}
	 \vert k \vert^2 \widehat{U}_k (t) + \frac{\vert k \vert^2}{\beta + \vert k \vert^2} \int_0^t \vert k \vert^2 \widehat{U}_{k}(s) (t-s)   \widehat{ \mu } (k(t - s)) \ ds &=\vert k \vert^2 \widehat{T}_k (t).
\end{align*}
We observe that we have an equation similar to (\ref{U_hat_1}) as we had for the linear case, except that our source term is more complicated.  Nevertheless, doing the same analysis as we did in the linear case we can obtain a similar equation to (\ref{F_U_2}), i.e.
\begin{equation*}
	F[\Delta U](t,z)   \leq F[\Delta T](t,z)  + C   \int_0^t  e^{-\frac{\theta_1}{4}  (t-s)}  F[\Delta T](s,z)  \ ds.
\end{equation*}
Then for $z= \lambda(t)$ and using the fact that the generator function is increasing in $z$ and for $s<t$ we have $\lambda(t) < \lambda(s)$,  we get  
\begin{equation}
\label{F_k^2 U}
	F[\Delta U](t,\lambda(t))  \leq  F[\Delta T](t,\lambda(t))  + C  \int_0^t  e^{-\frac{\theta_1}{4}  (t-s)}  F[\Delta T](s,\lambda(s))  \ ds.
\end{equation}
We need the following result to evaluate the right-hand side of (\ref{F_k^2 U}).
\begin{lem}
\label{lem F[k^2 T]}
With the same assumptions as in Proposition \ref{prop_bdd_F_U}, there holds
\begin{align}
\label{F_k^2 T}
	F[\Delta T](t,\lambda(t)) \leq C \sqrt{ \varepsilon } e^{- \frac{\lambda_1}{2} \langle t \rangle^{\gamma}}  +  C  F[h(U)](t,\lambda(t)) + C \sqrt{ \varepsilon }  \langle t \rangle^{- \sigma + 1}  \sup_{s\leq t} F[\Delta U](s,\lambda(s)) \langle s \rangle^{\sigma - 1}.
\end{align}
\end{lem}

\begin{proof}First, we multiply (\ref{T_hat}) by $\vert k \vert^2$,  take the absolute value and rewrite the third term as follow:
\begin{align*}
	\vert k \vert^2 \vert  \widehat{T}_k(t) \vert & \leq \frac{\vert k \vert^2}{\beta + \vert k \vert^2} \vert \widehat{f^0}_{k,kt} \vert + \frac{\vert k \vert^2}{\beta + \vert k \vert^2} \vert  \widehat{h(U)}_k (t) \vert +   \frac{\vert k \vert^2}{\beta + \vert k \vert^2} \int_0^t \sum_{\ell \in \mathbb{Z}^3 \setminus \{0\}}  \frac{ \vert k \vert(t  - s)}{\vert \ell \vert} \vert \ell \vert^2  \vert  \widehat{U}_{\ell}(s) \vert  \vert \widehat{g}_{k - \ell ,kt  - \ell s} (s)\vert \ ds.
\end{align*}
Next,  we use the notation $A_{k,\eta} (t) =  e^{\lambda (t) \langle k, \eta \rangle^{\gamma}} \langle k, \eta \rangle^{\sigma}$ and we compute $F[\Delta T](t,\lambda (t))$,
\begin{multline*}
	 \sup_{k \in \mathbb{Z}^3 \setminus \{0\}}  A_{k,kt} (t) \vert k \vert^2  \vert  \widehat{T}_k (t) \vert \frac{1}{\vert k \vert^{ \alpha }}   \leq  \sup_{k \in \mathbb{Z}^3 \setminus \{0\}}  A_{k,kt} (t) \vert \widehat{f^0}_{k,kt} \vert \frac{1}{\vert k \vert^{ \alpha }} +  \sup_{k \in \mathbb{Z}^3 \setminus \{0\}}  A_{k,kt} (t)   \vert \widehat{h(U)}_k (t) \vert \frac{1}{\vert k \vert^{ \alpha }} \\
	   +   \sup_{k \in \mathbb{Z}^3 \setminus \{0\}}  A_{k,kt} (t) \frac{1}{\vert k \vert^{ \alpha }}  \int_0^t \sum_{\ell \in \mathbb{Z}^3 \setminus \{0\}}  \frac{\vert k \vert(t  - s)}{\vert \ell \vert} \vert \ell \vert^2 \vert \widehat{U}_{\ell}(t)  \vert \vert \widehat{g}_{k - \ell ,kt  - \ell s} (s) \vert \ ds  \\
	 =: \mathcal{I}(t) + \mathcal{H}(t) + \mathcal{R}(t).
\end{multline*}
We start by evaluating $\mathcal{I}(t)$.  Since $\lambda(t) \leq 2 \lambda_0 < \frac{\lambda_1}{2}$ we have 
\begin{equation*}
	e^{\lambda(t) \langle k, kt \rangle^{\gamma}} = e^{\lambda_1 \langle k, kt \rangle^{\gamma}} e^{-(\lambda_1 - \lambda (t)) \langle k, kt \rangle^{\gamma}} \leq e^{\lambda_1 \langle k, kt \rangle^{\gamma}}  e^{-(\lambda_1 - \lambda (t)) \langle t \rangle^{\gamma}} \leq e^{\lambda_1 \langle k, kt \rangle^{\gamma}}  e^{-\frac{\lambda_1 }{2} \langle t \rangle^{\gamma}}.
\end{equation*}
Therefore, we can bound $\mathcal{I}(t)$ as follows:
\begin{multline*}
	\mathcal{I}(t)  := \sup_{k \in \mathbb{Z}^3 \setminus \{0\}} e^{\lambda (t) \langle k, kt \rangle^{\gamma}}   \vert \widehat{f^0}_{k,kt} \vert   \langle k, kt \rangle^{ \sigma} \frac{1}{\vert k \vert^{ \alpha }} \leq  e^{- \frac{\lambda_1}{2} \langle t \rangle^{\gamma}}  \sup_{k \in \mathbb{Z}^3 \setminus \{0\}} e^{ \lambda_1 \langle k, kt \rangle^{\gamma}}  \vert \widehat{f^0}_{k,kt} \vert   \langle k, kt \rangle^{\sigma} \frac{1}{\vert k \vert^{ \alpha }} \\
	\leq  e^{- \frac{\lambda_1}{2} \langle t \rangle^{\gamma}}  F[f^0](t,\lambda_1)   \leq C e^{-  \frac{\lambda_1}{2} \langle t \rangle^\gamma}  G[f^0]^\frac{1}{2}(\lambda_1) \leq C \sqrt{\varepsilon} e^{- \frac{\lambda_1}{2} \langle t \rangle^{\gamma}},
\end{multline*}
where we applied Lemma \ref{lemma_F<G} for the penultimate inequality and assumption (\ref{IC}) for the last one.

For $ \mathcal{H}(t) $, we  have 
\begin{equation*}
	\mathcal{H}(t) :=  \sup_{k \in \mathbb{Z}^3 \setminus \{0\}}  A_{k,kt} (t)   \vert \widehat{h(U)}_k (t) \vert \frac{1}{\vert k \vert^{ \alpha }}    =  F[h(U)](t,\lambda(t)). 
\end{equation*}

The last term $\mathcal{R}(t)$ is a bit more complicated. We follow the same analysis done in \cite{GNR} to treat this term.  This give us
\begin{align*}
	\mathcal{R}(t)  : & =  \sup_{k \in \mathbb{Z}^3  \setminus \{0\}} A_{k,kt} (t)\frac{1}{\vert k \vert^{ \alpha }}  \int_0^t  \sum_{\ell \in \mathbb{Z}^3 \setminus \{0\}} \frac{\vert k \vert (t  - s)}{\vert \ell \vert} \vert \ell \vert^2 \left| \widehat{U}_{\ell} (s)  \right| \vert  \widehat{g}_{k - \ell ,kt  - \ell s} (s) \vert \ ds  \\
	 &=  \sup_{k \in \mathbb{Z}^3  \setminus \{0\}}  \int_0^t \sum_{\ell \in \mathbb{Z}^3 \setminus \{0\}}  e^{\lambda(t) \langle k,  kt \rangle^{\gamma}} \langle k, kt \rangle^{\sigma}  \frac{\vert k \vert^{1 - \alpha} (t  - s)}{\vert \ell \vert} \vert \ell \vert^2 \left|  \widehat{U}_{\ell} (s)  \right| \vert  \widehat{g}_{k - \ell ,kt  - \ell s} (s) \vert \ ds  \\
	 & \leq  \sup_{k \in \mathbb{Z}^3  \setminus \{0\}}   \int_0^t \sum_{\ell \in \mathbb{Z}^3 \setminus \{0\}}  C_{k, \ell} (t,s)  e^{\lambda(s) \langle \ell,  \ell s \rangle^{\gamma}} \vert \ell \vert^2 \vert  \widehat{U}_{\ell} (s) \vert  \langle \ell,  \ell s \rangle^{\sigma} \frac{1}{\vert \ell \vert^\alpha}   \\
	& \qquad \times e^{\lambda(s) \langle k - \ell,  kt - \ell s \rangle^{\gamma}} \vert  \widehat{g}_{k - \ell ,kt  - \ell s} (s) \vert \langle k - \ell,  kt - \ell s \rangle^{\sigma} \ ds ,
\end{align*}
where we have used 
\begin{align*}
	e^{\lambda(t) \langle k,  kt \rangle^{\gamma}} & \leq e^{(\lambda(t) - \lambda(s)) \langle k,  kt \rangle^{\gamma}} e^{\lambda(s) \langle k,  kt \rangle^{\gamma}} \\
	& \leq e^{(\lambda(t) - \lambda(s)) \langle k,  kt \rangle^{\gamma}} e^{\lambda(s) \langle \ell,  \ell s \rangle^{\gamma}} e^{\lambda(s) \langle k - \ell,  kt - \ell s \rangle^{\gamma}},
\end{align*}
and 
\begin{equation*}
	C_{k, \ell} (t,s) := \frac{\vert k \vert^{1 - \alpha} (t-s)}{\vert \ell \vert^{1 - \alpha}} e^{(\lambda(t) - \lambda(s)) \langle k,  kt \rangle^{\gamma}} \langle k, kt \rangle^{\sigma} \langle k - \ell,  kt - \ell s \rangle^{-\sigma} \langle \ell,  \ell s \rangle^{-\sigma}.
\end{equation*}
Then, using inequality (\ref{inequaltiy_g}), we have 
\begin{align*}
	\mathcal{R}(t) &  \leq  \sup_{k \in \mathbb{Z}^3  \setminus \{0\}}   \int_0^t \sum_{\ell \in \mathbb{Z}^3 \setminus \{0\}}  C_{k, \ell} (t,s)  e^{\lambda(s) \langle \ell,  \ell s \rangle^{\gamma}} \vert \ell \vert^2 \vert \widehat{U}_{\ell} (s) \vert  \langle \ell,  \ell s \rangle^{\sigma} \frac{1}{\vert \ell \vert^\alpha}   \\
	& \qquad \qquad \qquad \times \sup_{\eta \in \br^d}  e^{\lambda(s) \langle k - \ell,  \eta \rangle^{\gamma}} \vert  \widehat{g}_{k - \ell ,\eta} (s) \vert \langle k - \ell,  \eta \rangle^{\sigma} \ ds \\
	&  \leq   C  \sup_{k \in \mathbb{Z}^3  \setminus \{0\}}   \int_0^t \sup_{\ell \in \mathbb{Z}^3  \setminus \{0\}} \left( e^{\lambda(s) \langle \ell,  \ell s \rangle^{\gamma}} \vert \ell \vert^2 \vert \widehat{U}_{\ell} (s) \vert  \langle \ell,  \ell s \rangle^{\sigma} \frac{1}{\vert \ell \vert^\alpha} \right) G[g(s)]^\frac{1}{2}(\lambda(s))  \sum_{\ell \in \mathbb{Z}^3 \setminus \{0\}}  C_{k, \ell} (t,s)     \ ds.
\end{align*}
Therefore, using the definition of the generator function (\ref{GenF}), we obtain
\begin{align*}
	\mathcal{R}(t) &  \leq   C  \sup_{k \in \mathbb{Z}^3  \setminus \{0\}}   \int_0^t F[\Delta U] (s, \lambda(s)) G[g(s)]^\frac{1}{2}(\lambda(s))  \sum_{\ell \in \mathbb{Z}^3 \setminus \{0\}}  C_{k, \ell} (t,s)   \ ds \\
	 & \leq C \sup_{s \leq t}  G[g(s)]^\frac{1}{2}(\lambda(s)))  \sup_{k \in \mathbb{Z}^3  \setminus \{0\}}   \int_0^t F[\Delta U] (s, \lambda(s)) \frac{\langle s \rangle^{\sigma - 1} }{\langle s \rangle^{\sigma - 1} } \sum_{\ell \in \mathbb{Z}^3 \setminus \{0\}}  C_{k, \ell} (t,s)   \ ds \\
	&  \leq 4 C_2 \sqrt{\varepsilon} \left( \sup_{s \leq t} F[\Delta U] (s, \lambda(s)) \langle s \rangle^{\sigma - 1} \right)  \sup_{k \in \mathbb{Z}^3  \setminus \{0\}}   \sum_{\ell \in \mathbb{Z}^3 \setminus \{0\}} \int_0^t C_{k, \ell} (t,s)  \langle s \rangle^{ - \sigma + 1}  \ ds,
\end{align*}
where we used assumption (\ref{assumption G<2eps}) and Fubini to swap the sum and the integral.
Finally, using \cite[Claim (4.23)]{GNR}, i.e. 
\begin{align*}
	\sup_{k \in \mathbb{Z}^3 \setminus \{0\}}   \sum_{\ell \in \mathbb{Z}^3 \setminus \{0\}} \int_0^t C_{k, \ell} (t,s) \langle s \rangle^{- \sigma + 1} \ ds  \leq C \langle t \rangle^{- \sigma + 1},
\end{align*}
we get
\begin{align*}
	\mathcal{R}(t) & \leq 4 C \cdot C_2 \sqrt{\varepsilon} \left( \sup_{s \leq t} F[\Delta U] (s, \lambda(s)) \langle s \rangle^{\sigma - 1} \right)  \langle t \rangle^{- \sigma + 1}.
\end{align*}
\cite[Claim (4.23)]{GNR} corresponds to the analysis of the plasma echoes. It is precisely in the analysis of the kernel $C_{k.l}(t,s)$ that we obtain the condition on the Gevrey index $\gamma > \frac{1}{3}$.

Hence,  with the estimates on $\mathcal{I}(t)$, $ \mathcal{H}(t) $ and $\mathcal{R}(t) $ the result of the lemma follows
\begin{multline*}
 	F[\Delta T](t,\lambda (t))  \leq \mathcal{I}(t) + \mathcal{H}(t) + \mathcal{R}(t) \\
 	 \leq C \sqrt{\varepsilon} e^{- \frac{ \lambda_1}{2}  \langle t \rangle^\gamma} + C  F[h(U)](t,\lambda(t))  + 4 C \cdot C_2 \sqrt{\varepsilon} \left( \sup_{s \leq t} F[\Delta U] (s, \lambda(s)) \langle s \rangle^{\sigma - 1} \right)  \langle t \rangle^{- \sigma + 1} .
 \end{multline*} 
\end{proof}
Let us finish the proof of Proposition \ref{prop_bdd_F_U}.  Combining (\ref{F_k^2 U}) and (\ref{F_k^2 T}) we have,
\begin{multline}
\label{F_Delta U aux}
	 F[\Delta U](t,\lambda(t))   \leq \left(C \sqrt{\varepsilon} e^{- \frac{ \lambda_1}{2}  \langle t \rangle^\gamma} + C  F[h(U)](t,\lambda(t))  + 4 C \cdot C_2 \sqrt{\varepsilon} \left( \sup_{s \leq t} F[\Delta U] (s, \lambda(s)) \langle s \rangle^{\sigma - 1} \right)  \langle t \rangle^{- \sigma + 1} \right) \\
	+ C   \int_0^t  e^{-\frac{\theta_1}{4}  (t-s)} \left(  C \sqrt{\varepsilon} e^{- \frac{ \lambda_1}{2}  \langle s \rangle^\gamma} + C  F[h(U)](s,\lambda(s))  + 4 C \cdot C_2 \sqrt{\varepsilon} \left( \sup_{u \leq s} F[\Delta U] (u, \lambda(u)) \langle u \rangle^{\sigma - 1} \right)  \langle s \rangle^{- \sigma + 1}  \right) \ ds  \\
	 \leq C \sqrt{\varepsilon} e^{-\nu t^\gamma}    + 4 C \cdot C_2 \sqrt{\varepsilon} \langle t \rangle^{-\sigma + 1}  \sup_{s \leq t} F[\Delta U](s,\lambda(s)) \langle s \rangle^{\sigma - 1}  \\
	 +  C  F[h(U)](t,\lambda(t))  + C \int_0^t e^{-\frac{\theta_1}{4}  (t-s)}  F[h(U)](s,\lambda(s)) \ ds,
\end{multline}
where we have used the following results to bound the integral: 
\begin{equation}
\label{inequality M1}  
	\int_0^t e^{-\frac{1}{4} \theta_1 (t-s)} \langle s \rangle^{- \sigma + 1} \ ds \leq C \langle t \rangle^{-  \sigma + 1},
\end{equation}
and
\begin{equation}
\label{inequality M2}
	\int_0^t e^{-\frac{1}{4} \theta_1 (t-s)} e^{-  \frac{\lambda_1}{2} \langle s \rangle^{\gamma} } \ ds \leq C e^{-\nu t^\gamma},
\end{equation}
with $\nu = \min \{ \frac{1}{8} \theta_1, \frac{\lambda_1}{4} \}.$ (See Lemmas \ref{lemma_M1} and \ref{lemma_M2} below for a proof of these two inequalities.)

Next, we multiply both sides of (\ref{F_Delta U aux}) by $\langle t \rangle^{\sigma -1} $ and take the supremum over $t$,
\begin{multline*}
	\sup_{s \leq t} F[\Delta U](s,\lambda(s))  \langle s \rangle^{\sigma -1}   \leq C \sqrt{\varepsilon} \sup_{s \leq t}  e^{- \nu s^\gamma} \langle s \rangle^{\sigma -1}   + 4 C \cdot C_2 \sqrt{\varepsilon}  \sup_{s \leq t} F[\Delta U](s,\lambda(s)) \langle s \rangle^{\sigma - 1}   \\
	 +  C \sup_{s \leq t}  F[h(U)](s,\lambda(s)) \langle s \rangle^{\sigma -1}   + C \sup_{s \leq t} \langle s \rangle^{\sigma -1}   \int_0^s e^{-\frac{\theta_1}{4}  (s-u)} F[h(U)](u,\lambda(u))  \ du.
\end{multline*}
For the first term on the right-hand side, we have $C \sqrt{\varepsilon} \sup_{s \leq t}  e^{-\nu s^\gamma}  \langle s \rangle^{\sigma -1} \leq C \sqrt{\varepsilon}$. And for the last term, we have 
\begin{multline*}
	 \langle s \rangle^{\sigma -1}  \int_0^s e^{-\frac{\theta_1}{4}  (s-u)}  F[h(U)](u,\lambda(u))  \ du   =    \int_0^s e^{-\frac{\theta_1}{4}  (s-u)}  F[h(U)](u,\lambda(u)) \langle s \rangle^{\sigma -1} \frac{\langle u \rangle^{\sigma -1}}{\langle u \rangle^{\sigma -1}} \ du \\
	   \leq \sup_{u \leq s} F[h(U)](u,\lambda(u)) \langle u \rangle^{\sigma -1}  \int_0^s e^{-\frac{\theta_1}{4}  (s-u)}  \frac{\langle s \rangle^{\sigma -1}}{\langle u \rangle^{ \sigma -1}} \ du 	  \leq C \sup_{u \leq s} F[h(U)](u,\lambda(u)) \langle u \rangle^{\sigma -1},
\end{multline*}
where we applied Lemma \ref{lemma_M1}  for the last inequality. Thus, we get
\begin{align*}
	\sup_{s \leq t} F[\Delta U](s,\lambda(s)) \langle s \rangle^{\sigma -1}  & \leq C \sqrt{\varepsilon} + 4 C \cdot C_2 \sqrt{\varepsilon} \sup_{s \leq t} F[\Delta U](s,\lambda(s)) \langle s \rangle^{\sigma - 1}  +  C \sup_{s \leq t}  F[h(U)](s,\lambda(s)) \langle s \rangle^{\sigma -1}.
\end{align*}
Rearranging the terms and using Lemma \ref{property_F_2}, we obtain
\begin{multline*}
	\sup_{s \leq t} F[\Delta U](s,\lambda(s)) \langle s \rangle^{\sigma -1}   \leq \frac{C \sqrt{\varepsilon} }{1 - 4 C \cdot C_2 \sqrt{\varepsilon} } +  \frac{ C }{1 - 4 C \cdot C_2 \sqrt{\varepsilon}}   \sup_{s \leq t}  \widetilde{h} \Big(C F[U](s,\lambda(s)) \Big) \langle s \rangle^{\sigma - 1}  \\
	 \leq C_1 \sqrt{\varepsilon} +  C \sup_{s \leq t}  \widetilde{h} \Big(C F[ \Delta U](s,\lambda(s)) \Big) \langle s \rangle^{\sigma - 1} 
	 \leq C_1 \sqrt{\varepsilon} +  C \widetilde{h} \left(  \sup_{s \leq t} C F[\Delta U](s,\lambda(s)) \langle s \rangle^{\sigma - 1} \right),
\end{multline*}
where we can take the constant $C_1$ large enough such that the inequalities hold. We used the facts that $y\widetilde{h}(x) \leq \widetilde{h}(xy)$ for $y \geq 1$ and  $\widetilde{h}$ is an increasing function on $\br_+$. Using Lemma \ref{Assumptions on the initial datum}, we end up with the same system (\ref{system Y(t)}) as in the linear case. Namely, if we define $\overline{Y}(t) :=  \sup_{s \leq t} C F[\Delta U](s,\lambda(s)) \langle s \rangle^{\sigma -1}$ we have  for every $t \in [0,T]$, 
\begin{eqnarray*}
    \left \{
    \begin{array}{l}
   \overline{Y}(t) \leq  C_1 \sqrt{\varepsilon} + C \widetilde{h} ( \overline{Y}(t)), \\ 
    \overline{Y}(0) \leq C_1 \sqrt{\varepsilon}.
    \end{array}
    \right.
\end{eqnarray*}
Therefore, similarly to (\ref{system Y(t)}), we have for every $t \in [0,T]$
\begin{align*}
	\sup_{s \leq t} F[\Delta U](s,\lambda(s)) \langle s \rangle^{\sigma -1}  \leq  2C_1  \sqrt{\varepsilon}.
\end{align*}
Hence, the result of Proposition \ref{prop_bdd_F_U} follows:
\begin{align*}
	F[\Delta U](t,\lambda(t))   \leq  2 C_1   \sqrt{\varepsilon} \langle t \rangle^{-\sigma + 1},
\end{align*}
for every time $t \in [0,T].$
\end{proof}

To conclude our proof we use the following bootstrap argument given in \cite[Proposition 1.21]{Tao}.
\begin{prop}[Abstract bootstrap principle]
\label{bootstrap}
Let $I$ be a time interval, and for each $t \in I$ suppose we have two statements, a “hypothesis” $\textbf{H}(t)$ and a “conclusion” $\textbf{C}(t)$. Suppose we can verify the following four assertions:
\begin{itemize}
\item[a)] (Hypothesis implies conclusion) If $\textbf{H}(t)$ is true for some time $t \in I$ then $\textbf{C}(t)$ is also true for that time $t$.
\item[b)](Conclusion is stronger than hypothesis) If $\textbf{C}(t)$ is true for some $t \in I$, then $\textbf{H}(t')$ is true for all $t' \in I$ in a neighbourhood of $t$.
\item[c)] (Conclusion is closed) If $t_1, t_2, . . . $ is a sequence of times in $I$ which converges to another time $t \in I$, and $\textbf{C}(t_n)$ is true for all $t_n$, then $\textbf{C}(t)$ is true.
\item[d)](Base case) $\textbf{H}(t)$ is true for at least one time $t \in I$.
\end{itemize}
Then $\textbf{C}(t)$ is true for all $t \in I$.
\end{prop}

Using this Proposition \ref{bootstrap}, we have the following result
\begin{prop}
Let $\gamma \in \left( \frac{1}{3}, 1 \right]$,  $\sigma > 3 + \delta$ and let $\lambda_0 < \frac{\lambda_1}{4}$.  Then there exists a constant $C_2 >0$ such that
\begin{equation}
\label{bound G_Delta U}
	G[g(t)] (\lambda(t)) \leq  4 C_2 \varepsilon,
\end{equation}
for every time $t>0$.
\end{prop}
\begin{proof}
We proved on $[0,T]$ that there exist constants $C_1 >0$ and $C_2 >0$ such that
\begin{align}
\label{continuity argument 1}
	F[\Delta U](t,\lambda(t))   \leq  4 C_1   \sqrt{\varepsilon} \langle t \rangle^{-\sigma + 1} \quad \implies \quad G[g(t)] (\lambda(t)) \leq  4 C_2  \varepsilon,
\end{align}
and 
\begin{align}
\label{continuity argument 2}
	 G[g(t)] (\lambda(t)) \leq   4 C_2  \varepsilon \quad \implies \quad F[\Delta U](t,\lambda(t))   \leq  2C_1  \sqrt{\varepsilon} \langle t \rangle^{-\sigma + 1}.
\end{align}
We want to apply the bootstrap argument given in Proposition \ref{bootstrap}. Following the notation of this proposition, let $I = \br_+$, let $\textbf{H}(t)$ be the hypothesis,
\begin{equation*}
	\textbf{H}(t) : F[\Delta U](t,\lambda(t))   \leq  4 C_1   \sqrt{\varepsilon} \langle t \rangle^{-\sigma + 1},
\end{equation*}
and $\textbf{C}(t)$ be the conclusion,
\begin{equation*}
	\textbf{C}(t) : G[g(t)] (\lambda(t)) \leq  4 C_2  \varepsilon.
\end{equation*}
We have to verify the four assertions.  By \eqref{continuity argument 1},  assertion a) is true for $t \in [0,T]$.  By \eqref{continuity argument 2} we know that $F[\Delta U](t,\lambda(t))   \leq  2 C_1   \sqrt{\varepsilon} \langle t \rangle^{-\sigma + 1}$. Therefore, by continuity, there exists a time $t' > t$ such that the following holds,
\begin{equation*}
 	F[\Delta U](t',\lambda(t'))   \leq  4 C_1   \sqrt{\varepsilon} \langle t' \rangle^{-\sigma + 1}.
 \end{equation*} 
Hence, assertion b) is true.  By continuity of $\textbf{C}(t)$, assertion c) is satisfied as well. Finally, for assertion d), we know that $\textbf{H}(0)$ is true thanks to Lemma \ref{Assumptions on the initial datum}.
Therefore, by Proposition \ref{bootstrap}, we have our result.
\end{proof}

Below we prove the two inequalities (\ref{inequality M1}) and (\ref{inequality M2}).
\begin{lem}
\label{lemma_M1}
Let $\sigma > 3$ and $\theta_1 > 0$ then we have for every $t > 0$,
\begin{equation*}
	\int_0^t e^{-\frac{1}{4} \theta_1 (t-s)} \langle s \rangle^{- \sigma + 1} \ ds \leq C \langle t \rangle^{- \sigma + 1}.
\end{equation*}
\end{lem}

\begin{proof}
We split the integral into two parts and use the fact that $\langle t \rangle^{- \sigma + 1}$ is a decreasing function,
\begin{align*}
	\int_0^t e^{-\frac{1}{4} \theta_1 (t-s)} \langle s \rangle^{- \sigma + 1} \ ds  & =  \int_0^{\frac{t}{2}} e^{-\frac{1}{4} \theta_1 (t-s)} \langle s \rangle^{- \sigma + 1} \ ds + \int_{\frac{t}{2}}^t e^{-\frac{1}{4} \theta_1 (t-s)} \langle s \rangle^{- \sigma + 1} \ ds \\	
	& \leq   \int_0^{\frac{t}{2}} e^{-\frac{1}{4} \theta_1 (t-s)} \ ds +  \left\langle \frac{t}{2} \right\rangle^{- \sigma + 1}  \int_{\frac{t}{2}}^t e^{-\frac{1}{4} \theta_1 (t-s)} \ ds \\
	&  \leq \frac{t}{2} e^{-\frac{1}{4} \theta_1 \frac{t}{2}} +  \left\langle \frac{t}{2} \right\rangle^{- \sigma + 1}  \int_0^{\frac{t}{2}} e^{-\frac{1}{4} \theta_1 \tau} \ d\tau \\
	&  \leq C  \langle t \rangle^{- \sigma + 1} + C  \langle t \rangle^{- \sigma + 1}  \leq 2 C \langle t \rangle^{- \sigma + 1}.
\end{align*}
\end{proof}

\begin{lem}
\label{lemma_M2}
Let $\lambda_1 > 0$ and $\theta_1 > 0$ then we have for every $t > 0$,
\begin{equation*}
	\int_0^t e^{-\frac{1}{4} \theta_1 (t-s)} e^{- \frac{\lambda_1}{2} \langle s \rangle^\gamma } \ ds \leq C e^{-\nu t^\gamma},
\end{equation*}
where $\nu = \min \{ \frac{1}{8} \theta_1, \frac{\lambda_1}{4} \}.$
\end{lem}

\begin{proof}
First, we split the integral into two parts,
\begin{align*}
	\int_0^t e^{-\frac{1}{4} \theta_1 (t-s)} e^{- \frac{\lambda_1}{2} \langle s \rangle^\gamma } \ ds & =  \int_0^{\frac{t}{2}} e^{-\frac{1}{4} \theta_1 (t-s)} e^{- \frac{\lambda_1}{2} \langle s \rangle^\gamma } \ ds + \int_{\frac{t}{2}}^t e^{-\frac{1}{4} \theta_1 (t-s)} e^{- \frac{\lambda_1}{2} \langle s \rangle^\gamma } \ ds.
\end{align*}
We bound the first term as follows:
\begin{align*}
	\int_0^{\frac{t}{2}} e^{-\frac{1}{4} \theta_1 (t-s)} e^{- \frac{\lambda_1}{2} \langle s \rangle^\gamma } \ ds   & \leq \int_0^{\frac{t}{2}} e^{-\frac{1}{4} \theta_1 (t-s)} \ ds  =  \int_{\frac{t}{2}}^t e^{-\frac{1}{4} \theta_1 \tau} \ d\tau = \left[ -\frac{4}{\theta_1} e^{-\frac{1}{4} \theta_1 \tau} \right]_{\frac{t}{2}}^t \\ 
	& = \frac{4}{\theta_1} e^{-\frac{1}{4} \theta_1 \frac{t}{2}} -\frac{4}{\theta_1} e^{-\frac{1}{4} \theta_1 t} 	 \leq C e^{-\frac{1}{4} \theta_1 \frac{t}{2}}.
\end{align*}
And for the second term we have 
\begin{align*}
	\int_{\frac{t}{2}}^t e^{-\frac{1}{4} \theta_1 (t-s)} e^{- \frac{\lambda_1}{2} \langle s \rangle^\gamma } \ ds & \leq e^{- \frac{\lambda_1}{2} \langle \frac{t}{2} \rangle^\gamma} \int_{\frac{t}{2}}^t e^{-\frac{1}{4} \theta_1 (t-s)} \ ds \leq C e^{- \frac{\lambda_1}{2} \langle \frac{t}{2} \rangle^\gamma}.
\end{align*}
Therefore, we get
\begin{align*}
	\int_0^t e^{-\frac{1}{4} \theta_1 (t-s)} e^{- \frac{\lambda_1}{2} \langle s \rangle^\gamma } \ ds & \leq C e^{-\frac{1}{4} \theta_1 \frac{t}{2}} + C e^{- \frac{\lambda_1}{2} \langle \frac{t}{2} \rangle^\gamma} \leq C e^{-\nu t^\gamma},
\end{align*}
with $\nu = \min \{ \frac{1}{8} \theta_1, \frac{\lambda_1}{4} \}.$

\end{proof}

\subsection{Bound on \texorpdfstring{$F[\rho](t,z)$}{Lg} and asymptotic behaviour}
\label{section bound_rho}

In this section we explain how to get a bound on the generator function of $\rho$ and we show the scattering (\ref{scaterring}). 

Since (\ref{bound G_Delta U}) holds for all times $t>0$, we can use Lemma \ref{lemma_F<G}, to bound the generator function of $\rho$ with the one of $g$. That is 
 \begin{equation*}
 	F[\rho] (t,\lambda(t)) \leq  C G[g(t)]^{\frac{1}{2}} (\lambda(t)) \leq C \sqrt{\varepsilon}.
 \end{equation*}
By repeating the same analysis as in Section \ref{Asymptotic behaviour}, we obtain the exponential decay of the density $\rho$ and of the electric field $E$ in every $C^k(\bt^3)$ norm, with $k \in \bn$.

Now, we show the scattering estimate (\ref{scaterring}). Integrating (\ref{nlequation_g}) in time, we have
\begin{equation}
\label{g integral}
	 g(t,x,v) = f^0(x,v) - \int_0^t E(s,x+vs) \cdot \nabla_v \mu (v) \ ds - \int_0^t  E(s,x+vs) \cdot (\nabla_v g - s \nabla_x g) \ ds.
\end{equation}
We can show that the two integrals are absolutely bounded in the $G[\cdot](z)$ norm for $ z \leq \frac{\lambda_0}{2}$ if $\gamma \in ( \frac{1}{3}, 1)$ and $z \leq \min \{ \frac{\theta_0}{2}, \frac{\lambda_0}{2} \}$ if $\gamma = 1$. Indeed,
\begin{align*}
	I_1 &:= \int_0^\infty G[E(s,x+vs) \cdot \nabla_v \mu (v)] (z) \ ds \\
	 & = \int_0^\infty \sum_{\vert j \vert \leq 3} \sum_{k \in \mathbb{Z}^3} \int_{\mathbb{R}^3} e^{2 z \langle k, \eta  \rangle^{\gamma} } \vert \partial_{\eta}^j \left( \reallywidehat{E(s,x+vs) \cdot \nabla_v \mu (v)}_{k,\eta} (s) \right) \vert^2 \langle k, \eta \rangle ^{2 \sigma} \ d\eta \ ds.
\end{align*}
By recalling the two equalities (\ref{E_mu_big_hat}) and (\ref{E_mu}), we obtain
\begin{align*}	 
	 I_1  & = \int_0^\infty \sum_{\vert j \vert \leq 3} \sum_{k \in \mathbb{Z}^3} \int_{\mathbb{R}^3} e^{2 z \langle k, \eta  \rangle^{\gamma} } \left\vert \partial_{\eta}^j \left(\widehat{E}_k (s) \cdot  ik(t-s) \widehat{ \mu } (\eta - ks) \right) \right\vert^2 \langle k, \eta \rangle ^{2 \sigma} \ d\eta \ ds \\
	  & = \int_0^\infty \sum_{\vert j \vert \leq 3} \sum_{k \in \mathbb{Z}^3}  \int_{\mathbb{R}^3} e^{2 z \langle k, \eta  \rangle^{\gamma} }  \vert \widehat{E}_k (s) \vert^2\vert \partial_{\eta}^j \left( i(\eta -ks) \widehat{\mu } (\eta - ks) \right) \vert^2 \langle k, \eta \rangle ^{2 \sigma} \ d\eta \ ds \\
	  & = \int_0^\infty  \sum_{k \in \mathbb{Z}^3}  \int_{\mathbb{R}^3} e^{2 z \langle k, \eta  \rangle^{\gamma} }  \vert \widehat{E}_k (s) \vert^2\vert  \widehat{\mu } (\eta - ks) \vert^2 \langle k, \eta \rangle ^{2 \sigma} \ d\eta \ ds \\
	 & \qquad + \int_0^\infty \sum_{\vert j \vert \leq 3} \sum_{k \in \mathbb{Z}^3}  \int_{\mathbb{R}^3} e^{2 z \langle k, \eta  \rangle^{\gamma} }  \vert \widehat{E}_k (s) \vert^2 \vert  \eta -ks \vert^2 \vert \partial_{\eta}^j \widehat{\mu } (\eta - ks) \vert^2 \langle k, \eta \rangle^{2 \sigma} \ d\eta \ ds  =: I_{11} + I_{12}.
\end{align*}
We first give a bound on $I_{11}$.  By triangle inequalities, we have
\begin{align*}
	e^{2 z \langle k, \eta  \rangle^{\gamma} } \leq e^{2 z \langle k, ks  \rangle^{\gamma} } e^{2 z \langle 0, \eta - ks  \rangle^{\gamma} },
\end{align*}
and for some constant $C$
\begin{align*}
	\langle k, \eta \rangle ^{2 \sigma} \leq C \langle k,  ks \rangle ^{2 \sigma} \langle 0, \eta - ks \rangle ^{2 \sigma}.
\end{align*}
Therefore, using the two latter inequalities,
\begin{align*}
	I_{11} & \leq C  \int_0^\infty  \sum_{k \in \mathbb{Z}^3}  \int_{\mathbb{R}^3} e^{2 z \langle k, ks  \rangle^{\gamma} }  \vert \widehat{E}_k (s) \vert^2  \langle k,  ks \rangle ^{2 \sigma}  e^{2 z \langle 0, \eta - ks  \rangle^{\gamma} }  \vert  \widehat{\mu } (\eta - ks) \vert^2 \langle 0, \eta - ks \rangle ^{2 \sigma} \ d\eta \ ds \\
	& \leq C \int_0^\infty  \sum_{k \in \mathbb{Z}^3} e^{2 (z - \lambda (s)) \langle k, ks  \rangle^{\gamma} }  e^{2 \lambda (s) \langle k, ks  \rangle^{\gamma} }  \vert \widehat{E}_k (s) \vert^2  \langle k,  ks \rangle ^{2 \sigma} \int_{\mathbb{R}^3} e^{2 z \langle \eta - ks  \rangle^{\gamma} }  \vert  \widehat{\mu } (\eta - ks) \vert^2 \langle \eta - ks \rangle ^{2 \sigma} \ d\eta \ ds.
\end{align*}
We can bound the integral in $\eta$ using (\ref{Pmu1}) if $\gamma \in (\frac{1}{3}, 1)$ and if $\gamma = 1$ we need the additional assumption that $z \leq \frac{\theta_0}{2}$.  This give us
\begin{align*}
	I_{11} & \leq C \int_0^\infty  \sum_{k \in \mathbb{Z}^3} e^{2 (z - \lambda (s)) \langle k, ks  \rangle^{\gamma} }  e^{2 \lambda (s) \langle k, ks  \rangle^{\gamma} }  \vert \widehat{E}_k (s) \vert^2  \langle k,  ks \rangle ^{2 \sigma} \int_{\mathbb{R}^3} e^{2 z \langle \eta - ks  \rangle^{\gamma} }    e^{-2 \theta_0  \vert \eta - ks \vert} \langle \eta - ks \rangle ^{2 \sigma} \ d\eta \ ds \\
	& \leq  C \int_0^\infty e^{2 (z - \lambda (s)) \langle s  \rangle^{\gamma} }  \sum_{k \in \mathbb{Z}^3}   e^{2 \lambda (s) \langle k, ks  \rangle^{\gamma} }  \vert \widehat{E}_k (s) \vert^2  \langle k,  ks \rangle ^{2 \sigma} \ ds,
\end{align*}
where we used $\langle k, ks  \rangle^{\gamma}  \geq \langle s  \rangle^{\gamma} $ and $(z - \lambda(s))$ is always negative since $z \leq \frac{\lambda_0}{2} < \lambda (s)$. Therefore,  with $\vert \widehat{E}_{k} (s) \vert = \frac{1}{\vert k \vert} \vert k \vert^2 \vert \widehat{U}_k (s) \vert$ and (\ref{continuity argument 2}), we have
\begin{align*}
	I_{11} & \leq  C \int_0^\infty e^{2 (z - \lambda (s)) \langle s  \rangle^{\gamma} }  \left( \sup_{k \in \mathbb{Z}^3}   e^{ \lambda (s) \langle k, ks  \rangle^{\gamma} } \vert k \vert^2 \vert  \widehat{U}_k (s) \vert  \langle k,  ks \rangle ^{ \sigma} \frac{1}{\vert k \vert^\alpha} \right)^2 \sum_{k \in \mathbb{Z}^3}  \vert k \vert^{2 \alpha - 2}  \ ds \\
	& \leq  C \varepsilon  \int_0^\infty e^{2 (z - \lambda_0) \langle s  \rangle^{\gamma} }   \langle s \rangle^{-2 \sigma +2}   \ ds \leq C \varepsilon,
\end{align*}
provided that $\alpha < \frac{1}{2}$ and $z \leq \frac{\lambda_0}{2} < \lambda (s)$. Similarly, we can show that $I_{12}$ is bounded in the same way.

For the second integral, first recall that 
\begin{align*}
	\reallywidehat{E(s,x+vs) (\nabla_v g - s \nabla_x g)} (k, \eta)= i \sum_{\ell \in \mathbb{Z}^3} (\eta - ks) \cdot  \widehat{E}_{\ell} (s) \widehat{g}_{k - \ell ,\eta  - \ell s} (s).
\end{align*}
Thus, for $z \leq \frac{\lambda_0}{2}$, we get
\begin{align*}
	I_2 :&= \int_0^\infty G[E(s,x+vs) \cdot (\nabla_v g - s \nabla_x g)] (z) \ ds \\
	 & = \int_0^\infty \sum_{\vert j \vert \leq 3} \sum_{k \in \mathbb{Z}^3} \int_{\mathbb{R}^3} e^{2z\langle k, \eta  \rangle^{\gamma} }  \Bigg| \partial_{\eta}^j \Bigg(i \sum_{\ell \in \mathbb{Z}^3} (\eta - ks) \cdot \widehat{E}_{\ell} (s) \widehat{g}_{k - \ell ,\eta  - \ell s} (s) \Bigg)  \Bigg|^2 \langle k, \eta \rangle^{2 \sigma} \ d\eta \ ds \\
	  & = \int_0^\infty   \sum_{k \in \mathbb{Z}^3} \int_{\mathbb{R}^3} e^{2z\langle k, \eta  \rangle^{\gamma} }  \left| \sum_{\ell \in \mathbb{Z}^3}  \widehat{E}_{\ell} (s) \widehat{g}_{k - \ell ,\eta  - \ell s} (s)  \right|^2 \langle k, \eta \rangle ^{2 \sigma} \ d\eta \ ds \\
	   & \qquad + \int_0^\infty    \sum_{\vert j \vert \leq 3} \sum_{k \in \mathbb{Z}^3} \int_{\mathbb{R}^3} e^{2z\langle k, \eta  \rangle^{\gamma} } \vert \eta - ks \vert^2  \left|  \sum_{\ell \in \mathbb{Z}^3}  \widehat{E}_{\ell} (s) \partial_{\eta}^j \widehat{g}_{k - \ell ,\eta  - \ell s} (s)  \right|^2 \langle k, \eta \rangle ^{2 \sigma} \ d\eta \ ds =: I_{21} + I_{22}.
\end{align*}
To show that $I_{21}$ is bounded, we use the algebra property given by inequality (3.14) in \cite[Lemma 3.3]{BMM_13}:
\begin{align*}
	I_{21} & = \int_0^\infty   \sum_{k \in \mathbb{Z}^3} \int_{\mathbb{R}^3}   \left| \sum_{\ell \in \mathbb{Z}^3}  e^{z \langle k, \eta  \rangle^{\gamma} }  \langle k, \eta \rangle^{ \sigma}  \widehat{E}_{\ell} (s) \widehat{g}_{k - \ell ,\eta  - \ell s} (s)  \right|^2  \ d\eta \ ds \\
	& \leq \int_0^\infty \Bigg( \sum_{\ell \in \mathbb{Z}^3}  e^{2 z \langle \ell, \ell s  \rangle^{\gamma} }  \langle \ell, \ell s \rangle^{2 \sigma}  \vert \widehat{E}_{\ell} (s) \vert^2 \Bigg)  \sum_{k \in \mathbb{Z}^3} \int_{\mathbb{R}^3} e^{2 z \langle k, \eta  \rangle^{\gamma} } \langle k, \eta  \rangle^{2 \sigma}   \left|  \widehat{g}_{k ,\eta } (s)  \right|^2  \ d\eta \ ds.
\end{align*}
Then, we use that $z \leq \frac{\lambda_0}{2} < \lambda(s)$ for all $s$ and the fact that the generator function of $g$ is increasing in the $z$ variable. Therefore,
\begin{align*}
	I_{21} & \leq \int_0^\infty    \left( \sum_{\ell \in \mathbb{Z}^3} e^{2( z - \lambda(s))  \langle \ell, \ell s  \rangle^{\gamma} }   e^{2 \lambda(s) \langle \ell, \ell s  \rangle^{\gamma} }  \langle \ell, \ell s \rangle^{2 \sigma}  \vert \widehat{E}_{\ell} (s) \vert^2 \right) G[g(s)] (\lambda(s)) \ ds \\
	& \leq \int_0^\infty   \left( \sup_{\ell \in \mathbb{Z}^3} e^{ \lambda(s) \langle \ell, \ell s  \rangle^{\gamma} }  \langle \ell, \ell s \rangle^{ \sigma} \vert \ell \vert^2 \vert \widehat{U}_l (s) \vert \frac{1}{\vert \ell \vert^{\alpha}} \right)^2  \left( \sum_{\ell \in \mathbb{Z}^3} e^{2( z - \lambda(s)) \langle \ell, \ell s  \rangle^{\gamma} } \vert \ell \vert^{2 \alpha - 2}  \right) G[g(s)] ( \lambda(s)) \ ds,
\end{align*}
where we have used $\vert \widehat{E}_{\ell} (s) \vert = \frac{1}{\vert \ell \vert} \vert \ell \vert^2 \vert  \widehat{U}_\ell (s) \vert$. Then using (\ref{bound G_Delta U}) and (\ref{continuity argument 2}), we obtain
\begin{align*}
	I_{21} & \leq C \varepsilon^2 \int_0^\infty  \langle s \rangle^{- 2 \sigma + 2}  e^{2( z - \lambda(s)) \langle s  \rangle^{\gamma} }  \left( \sum_{\ell \in \mathbb{Z}^3}  \vert \ell \vert^{2 \alpha - 2}  \right)  \ ds  \leq C \varepsilon^2 \int_0^\infty  \langle s \rangle^{- 2 \sigma + 2}   e^{2( z - \lambda_0) \langle s  \rangle^{\gamma} }   \ ds \leq C \varepsilon^2,
\end{align*}
provided that $\alpha < \frac{1}{2}$ and $z \leq \frac{\lambda_0}{2} < \lambda (s)$. Similarly, we can show that $I_{22}$ is bounded in the same way.

Now, that we have shown that the two integrals are absolutely bounded, we can define  the limit $f_\infty$ as
\begin{align}
\label{f_infinity}
	f_\infty(x,v) := f^0(x,v) - \int_0^\infty E(s,x+vs) \cdot \nabla_v \mu (v) \ ds - \int_0^\infty  E(s,x+vs) \cdot  (\nabla_v g - s \nabla_x g) \ ds.
\end{align}
Since  $f(t,x+vt,v) = g(t,x,v)$, we have combining, (\ref{g integral}) and (\ref{f_infinity}), 
\begin{align*}
	G[f(t,x+vt, v) - f_\infty(x,v)](z) &  \leq \int_t^\infty G[E(s,x+vs) \cdot \nabla_v \mu (v)] (z) \ ds \\
	 & \qquad + \int_t^\infty G[E(s,x+vs) \cdot (\nabla_v g - s \nabla_x g)] (z) \ ds \\
	 & \leq C \varepsilon  \int_t^\infty e^{2 (z - \lambda_0) \langle s  \rangle^{\gamma} }   \langle s \rangle^{-2 \sigma +2}   \ ds + C \varepsilon^2 \int_t^\infty  \langle s \rangle^{- 2 \sigma + 2}   e^{2( z - \lambda_0) \langle s  \rangle^{\gamma} }   \ ds \\
	 & \leq C \varepsilon e^{2 (z - \lambda_0) \langle t  \rangle^{\gamma} } + C \varepsilon^2 e^{2( z - \lambda_0) \langle t  \rangle^{\gamma} } \leq C \varepsilon e^{- 2 (\lambda_0 - z) \langle t  \rangle^{\gamma} }.
\end{align*}
From this last inequality, we can deduce that $f(t,x,v)$ converges weakly to the spatial average of $f_\infty$. Indeed, by (\ref{scaterring}), we deduce 
\begin{align*}
	\vert \widehat{f}_{k,\eta - kt} (t) -  \widehat{f_{\infty}} (k,\eta) \vert  \xrightarrow[t \rightarrow \infty]{} 0,
\end{align*}
or equivalently,
\begin{align}
\label{decay_fourier_transform_f}
	\vert \widehat{f}_{k,\eta} (t) -  \widehat{f_{\infty}} (k,\eta + kt) \vert  \xrightarrow[t \rightarrow \infty]{} 0.
\end{align}
Let $\varphi \in C_c^0 (\bt^d \times \br^3)$ be a given test function.  And let us write the $L^2 (\bt^d \times \br^d)$ scalar product as
\begin{equation*}
	\langle \varphi (x,v),  f(t,x,v) \rangle_{L^2} :=  \int_{\bt^3 \times \br^3} \varphi (x,v) f(t,x,v) \ dx  \ dv.
\end{equation*}
Therefore, using Parseval's identity and Plancherel theorem, we have
\begin{align*}
	\langle \varphi (x,v),  f(t,x,v) \rangle_{L^2}  &= \langle  \widehat{\varphi }_{k,\eta}, \widehat{f}_{k,\eta}(t) \rangle_{L^2} = \langle  \widehat{\varphi }_{k,\eta} , \big(\widehat{f}_{k,\eta} (t) -  \widehat{f_{\infty}} (k,\eta + kt) \big) \rangle_{L^2}  + \langle  \widehat{\varphi }_{k,\eta} , \widehat{f_{\infty}} (k,\eta + kt) \rangle_{L^2}.
\end{align*}
The first term on the right-hand side goes to zero thanks to (\ref{decay_fourier_transform_f}), and we can treat the second term on the right-hand side as in the proof of Proposition \ref{prop weak convergence}. Therefore,  we obtain the weak convergence in $L^2 (\bt^d \times \br^d)$
\begin{equation*}
	f(t,x,v) \xrightharpoonup[t \rightarrow \infty]{} \langle f_\infty (v) \rangle_x = \int_{\mathbb{T}^3} f_\infty(x,v) \ dx.
\end{equation*}

\textbf{Acknowledgements:} The authors would like to warmly thank the anonymous referee for the constructive review and the useful comments that improved the presentation of the results.

\bibliographystyle{plain}
\bibliography{biblio}

\begin{thebibliography}{10}

\bibitem{Arsenev}
A.~A. Arsenev.
\newblock Existence in the large of a weak solution of {V}lasov's system of
  equations.
\newblock {\em \v{Z}. Vy\v{c}isl. Mat i Mat. Fiz.}, 15:136--147, 276, 1975.

\bibitem{Bardos_Degond_1}
C.~Bardos and P.~Degond.
\newblock Existence globale des solutions des \'{e}quations de
  {V}lasov-{P}oisson.
\newblock In {\em Nonlinear partial differential equations and their
  applications. {C}oll\`ege de {F}rance seminar, {V}ol. {VII} ({P}aris,
  1983--1984)}, volume 122 of {\em Res. Notes in Math.}, pages 1--3, 35--58.
  Pitman, Boston, MA, 1985.

\bibitem{Bardos_Degond_2}
C.~Bardos and P.~Degond.
\newblock Global existence for the {V}lasov-{P}oisson equation in {$3$} space
  variables with small initial data.
\newblock {\em Ann. Inst. H. Poincar\'{e} Anal. Non Lin\'{e}aire},
  2(2):101--118, 1985.

\bibitem{Bardos_Degond_Golse}
C.~Bardos, P.~Degond, and F.~Golse.
\newblock A priori estimates and existence results for the {V}lasov and
  {B}oltzmann equations.
\newblock In {\em Nonlinear systems of partial differential equations in
  applied mathematics, {P}art 2 ({S}anta {F}e, {N}.{M}., 1984)}, volume~23 of
  {\em Lectures in Appl. Math.}, pages 189--207. Amer. Math. Soc., Providence,
  RI, 1986.

\bibitem{Bardos}
C.~Bardos, F.~Golse, T.~Nguyen, and R.~Sentis.
\newblock The {M}axwell-{B}oltzmann approximation for ion kinetic modeling.
\newblock {\em Phys. D}, 376/377:94--107, 2018.

\bibitem{Batt_Rein}
J.~Batt and G.~Rein.
\newblock Global classical solutions of the periodic {V}lasov-{P}oisson system
  in three dimensions.
\newblock {\em C. R. Acad. Sci. Paris S\'{e}r. I Math.}, 313(6):411--416, 1991.

\bibitem{JB_collision}
J.~Bedrossian.
\newblock Suppression of plasma echoes and {L}andau damping in {S}obolev spaces
  by weak collisions in a {V}lasov-{F}okker-{P}lanck equation.
\newblock {\em Ann. PDE}, 3(2):Paper No. 19, 66, 2017.

\bibitem{JB}
J.~Bedrossian.
\newblock Nonlinear echoes and {L}andau damping with insufficient regularity.
\newblock {\em Tunis. J. Math.}, 3(1):121--205, 2021.

\bibitem{BM}
J.~Bedrossian and N.~Masmoudi.
\newblock Inviscid damping and the asymptotic stability of planar shear flows
  in the 2{D} {E}uler equations.
\newblock {\em Publ. Math. Inst. Hautes \'{E}tudes Sci.}, 122:195--300, 2015.

\bibitem{BMM_13}
J.~Bedrossian, N.~Masmoudi, and C.~Mouhot.
\newblock Landau damping: paraproducts and {G}evrey regularity.
\newblock {\em Ann. PDE}, 2(1):Art. 4, 71, 2016.

\bibitem{BMM_16}
J.~Bedrossian, N.~Masmoudi, and C.~Mouhot.
\newblock Landau damping in finite regularity for unconfined systems with
  screened interactions.
\newblock {\em Comm. Pure Appl. Math.}, 71(3):537--576, 2018.

\bibitem{BMM_22_linear_VP}
J.~Bedrossian, N.~Masmoudi, and C.~Mouhot.
\newblock Linearized {W}ave-{D}amping {S}tructure of {V}lasov-{P}oisson in
  {$\mathbb{R}^3$}.
\newblock {\em SIAM J. Math. Anal.}, 54(4):4379--4406, 2022.

\bibitem{Bouchut}
F.~Bouchut.
\newblock Global weak solution of the {V}lasov-{P}oisson system for small
  electrons mass.
\newblock {\em Comm. Partial Differential Equations}, 16(8-9):1337--1365, 1991.

\bibitem{plasma_2003}
T.~J.~M. Boyd and J.~J. Sanderson.
\newblock {\em The Physics of Plasmas}.
\newblock Cambridge University Press, 2003.

\bibitem{Cesbron_Iacobelli}
L.~Cesbron and M.~Iacobelli.
\newblock Global well-posedness of {V}lasov-{P}oisson-type systems in bounded
  domains, 08 2021.
\newblock arXiv:2108.11209v1.

\bibitem{Nguyen_collision}
S.~Chaturvedi, J.~Luk, and T.~Nguyen.
\newblock The {V}lasov-{P}oisson-{L}andau system in the weakly collisional
  regime.
\newblock 04 2021.
\newblock arXiv:2104.05692.

\bibitem{Chen}
Z.~Chen and J.~Chen.
\newblock Moments propagation for weak solutions of the {V}lasov-{P}oisson
  system in the three-dimensional torus.
\newblock {\em J. Math. Anal. Appl.}, 472(1):728--737, 2019.

\bibitem{EG_TN}
E.~Grenier and T.~Nguyen.
\newblock Generator functions and their applications.
\newblock {\em Proc. Amer. Math. Soc. Ser. B}, 8:245--251, 2021.

\bibitem{GNR}
E.~Grenier, T.~Nguyen, and I.~Rodnianski.
\newblock Landau damping for analytic and {G}evrey data.
\newblock {\em Math. Res. Lett.}, 28(6):1679--1702, 2021.

\bibitem{Griffin_Iacobelli_stability}
M.~Griffin-Pickering and M.~Iacobelli.
\newblock Singular limits for plasmas with thermalised electrons.
\newblock {\em J. Math. Pures Appl. (9)}, 135:199--255, 2020.

\bibitem{Griffin_Iacobelli_R}
M.~Griffin-Pickering and M.~Iacobelli.
\newblock Global strong solutions in {$\mathbb{R}^3$} for ionic
  {V}lasov-{P}oisson systems.
\newblock {\em Kinet. Relat. Models}, 14(4):571--597, 2021.

\bibitem{Griffin_Iacobelli_torus}
M.~Griffin-Pickering and M.~Iacobelli.
\newblock Global well-posedness for the {V}lasov-{P}oisson system with massless
  electrons in the 3-dimensional torus.
\newblock {\em Comm. Partial Differential Equations}, 46(10):1892--1939, 2021.

\bibitem{Griffin_Iacobelli_summary}
M.~Griffin-Pickering and M.~Iacobelli.
\newblock Recent developments on the well-posedness theory for {V}lasov-type
  equations.
\newblock In {\em From particle systems to partial differential equations},
  volume 352 of {\em Springer Proc. Math. Stat.}, pages 301--319. Springer,
  Cham, 2021.

\bibitem{gravitational_LD}
M.~Had\v{z}i\'{c}, G.~Rein, M~Schrecker, and C.~Straub.
\newblock Damping versus oscillations for a gravitational {V}lasov-{P}oisson
  system, 01 2023.
\newblock arXiv:2301.07662.

\bibitem{HKD_thesis}
D.~Han-Kwan.
\newblock {\em {Contribution {\`a} l'{\'e}tude math{\'e}matique des plasmas
  fortement magn{\'e}tis{\'e}s}}.
\newblock Theses, {Universit{\'e} Pierre et Marie Curie - Paris VI}, July 2011.

\bibitem{HKD}
D.~Han-Kwan.
\newblock Quasineutral limit of the {V}lasov-{P}oisson system with massless
  electrons.
\newblock {\em Comm. Partial Differential Equations}, 36(8):1385--1425, 2011.

\bibitem{HKD_HDR}
D.~Han-Kwan.
\newblock Stabilit\'e, limites singuli\`eres et conditions de contr\^ole
  g\'eom\'etrique en th\'eorie cin\'etique, 09 2017.
\newblock M\'emoire pr\'esente\'e \`a Universit\'e Paris-Diderot pour
  l’obtention de l’habilitation \`a diriger des recherches.

\bibitem{Han-Kwan_Iacobelli}
D.~Han-Kwan and M.~Iacobelli.
\newblock The quasineutral limit of the {V}lasov-{P}oisson equation in
  {W}asserstein metric.
\newblock {\em Commun. Math. Sci.}, 15(2):481--509, 2017.

\bibitem{HKNR}
D.~Han-Kwan, T.~Nguyen, and F.~Rousset.
\newblock Asymptotic stability of equilibria for screened {V}lasov-{P}oisson
  systems via pointwise dispersive estimates.
\newblock {\em Ann. PDE}, 7(2):Paper No. 18, 37, 2021.

\bibitem{HKNR_linear_VP}
D.~Han-Kwan, T.~Nguyen, and F.~Rousset.
\newblock On the linearized {V}lasov-{P}oisson system on the whole space around
  stable homogeneous equilibria.
\newblock {\em Comm. Math. Phys.}, 387(3):1405--1440, 2021.

\bibitem{Horst_Hunze}
E.~Horst and R.~Hunze.
\newblock Weak solutions of the initial value problem for the unmodified
  nonlinear {V}lasov equation.
\newblock {\em Math. Methods Appl. Sci.}, 6(2):262--279, 1984.

\bibitem{huang_2d}
L.~Huang, Q.-H. Nguyen, and Y.~Xu.
\newblock Nonlinear {Landau} damping for the 2d {Vlasov}-{Poisson} system with
  massless electrons around {Penrose}-stable equilibria.
\newblock 06 2022.
\newblock arXiv:2206.11744.

\bibitem{huang_sharp_2022}
L.~Huang, Q.-H. Nguyen, and Y.~Xu.
\newblock Sharp estimates for screened {Vlasov}-{Poisson} system around
  {Penrose}-stable equilibria in $\mathbb{R}^d$, $d\geq3$, 05 2022.
\newblock arXiv:2205.10261.

\bibitem{RHRW}
R.~Höfer and R.~Winter.
\newblock A fast point charge interacting with the screened {V}lasov-{P}oisson
  system.
\newblock 05 2022.
\newblock arXiv:2205.00035.

\bibitem{IPWW}
A.~Ionescu, B.~Pausader, X.~Wang, and K.~Widmayer.
\newblock Nonlinear landau damping for the {V}lasov-{P}oisson system in
  {$\mathbb{R}^3$}: the {P}oisson equilibrium.
\newblock 05 2022.
\newblock arXiv:2205.04540.

\bibitem{Iordanskii}
S.~V. Iordanski\u{\i}.
\newblock The {C}auchy problem for the kinetic equation of plasma.
\newblock {\em Trudy Mat. Inst. Steklov.}, 60:181--194, 1961.

\bibitem{Galactic_2008}
B.~James and T.~Scott.
\newblock {\em Galactic Dynamics: Second Edition}.
\newblock Princeton University Press, 2008.

\bibitem{Landau}
L.~Landau.
\newblock On the vibrations of the electronic plasma.
\newblock {\em (Russian) Akad. Nauk SSSR. Zhurnal Eksper. Teoret. Fiz},
  16:574--586, 1946.

\bibitem{Lions_Perthame}
P.-L. Lions and B.~Perthame.
\newblock Propagation of moments and regularity for the {$3$}-dimensional
  {V}lasov-{P}oisson system.
\newblock {\em Invent. Math.}, 105(2):415--430, 1991.

\bibitem{Loeper}
G.~Loeper.
\newblock Uniqueness of the solution to the {V}lasov-{P}oisson system with
  bounded density.
\newblock {\em J. Math. Pures Appl. (9)}, 86(1):68--79, 2006.

\bibitem{CM_CV}
C.~Mouhot and C.~Villani.
\newblock On {L}andau damping.
\newblock {\em Acta Math.}, 207(1):29--201, 2011.

\bibitem{Pallard}
C.~Pallard.
\newblock Moment propagation for weak solutions to the {V}lasov-{P}oisson
  system.
\newblock {\em Comm. Partial Differential Equations}, 37(7):1273--1285, 2012.

\bibitem{Penrose}
O.~Penrose.
\newblock Electrostatic instabilities of a uniform non-maxwellian plasma.
\newblock {\em Physics of Fluids (U.S.)}, Vol: 3, 3 1960.

\bibitem{Pfaffelmoser}
K.~Pfaffelmoser.
\newblock Global classical solutions of the {V}lasov-{P}oisson system in three
  dimensions for general initial data.
\newblock {\em J. Differential Equations}, 95(2):281--303, 1992.

\bibitem{Ryutov_1999}
D.~D. Ryutov.
\newblock Landau damping: half a century with the great discovery.
\newblock {\em Plasma Physics and Controlled Fusion}, 41(3A):A1--A12, jan 1999.

\bibitem{Schaffer}
J.~Schaeffer.
\newblock Global existence of smooth solutions to the {V}lasov-{P}oisson system
  in three dimensions.
\newblock {\em Comm. Partial Differential Equations}, 16(8-9):1313--1335, 1991.

\bibitem{Tao}
T.~Terence.
\newblock {\em Nonlinear dispersive equations}, volume 106 of {\em CBMS
  Regional Conference Series in Mathematics}.
\newblock Conference Board of the Mathematical Sciences, Washington, DC; by the
  American Mathematical Society, Providence, RI, 2006.
\newblock Local and global analysis.

\bibitem{IT}
I.~Tristani.
\newblock Landau damping for the linearized {V}lasov {P}oisson equation in a
  weakly collisional regime.
\newblock {\em J. Stat. Phys.}, 169(1):107--125, 2017.

\bibitem{Ukai_Okabe}
S.~Ukai and T.~Okabe.
\newblock On classical solutions in the large in time of two-dimensional
  {V}lasov's equation.
\newblock {\em Osaka Math. J.}, 15(2):245--261, 1978.

\bibitem{Villani_notes}
C.~Villani.
\newblock Landau damping, {N}otes de cours.
\newblock CEMRACS, 2010.
\newblock
  \url{http://www.cedricvillani.org/sites/dev/files/old_images/2012/08/B13.Landau.pdf}.

\end{thebibliography}

\end{document}